\documentclass[11pt,reqno]{amsart}

\usepackage{amssymb}
\usepackage{amssymb}
\usepackage{amsmath}
\usepackage{graphicx}

\newcommand{\C}{\mathbb{C}}

\newcommand{\N}{\mathbb{N}}
\newcommand{\R}{\mathbb{R}}

\newcommand{\OO}{\mathcal O}
\newcommand{\eps}{\varepsilon}

\DeclareMathOperator{\spec}{sp}

\newtheorem{theorem}{Theorem}[section]
\newtheorem{lemma}[theorem]{Lemma}
\newtheorem{corollary}[theorem]{Corollary}
\newtheorem{proposition}[theorem]{Proposition}
\newtheorem{definition}[theorem]{Definition}
\newtheorem{Remark}[theorem]{Remark}
\newtheorem{Example}[theorem]{Example}

\newenvironment{example}{\begin{Example}\rm}{\end{Example}}
\numberwithin{equation}{section}

\title{An equilibrium problem for the limiting eigenvalue distribution of
banded Toeplitz matrices}
\author{Maurice Duits \and Arno B.J. Kuijlaars}
\thanks{Department of Mathematics, Katholieke Universiteit Leuven,
 Celestijnenlaan 200B, 3001 Leuven, Belgium.
 (maurice.duits@wis.kuleuven.be,
arno.kuijlaars@wis.kuleuven.be). The first author is a research
assistant of the Fund for Scientific Research -- Flanders. The
authors were supported by the European Science Foundation Program
MISGAM. The second author is supported by FWO-Flanders project
G.0455.04, by K.U. Leuven research grant OT/04/21, by Belgian
Interuniversity Attraction Pole NOSY P06/02, and by a grant from the
Ministry of Education and Science of Spain, project code
MTM2005-08648-C02-01. }

\begin{document}

\maketitle

\begin{abstract} We study the limiting eigenvalue distribution of
$n\times n$ banded Toeplitz matrices as $n\to \infty$. From
classical results of Schmidt-Spitzer and Hirschman it is known that
the eigenvalues accumulate on a special curve in the complex plane
and the normalized eigenvalue counting measure converges weakly to a
measure on this curve as $n\to\infty$. In this paper, we
characterize the limiting measure in terms of an equilibrium
problem. The limiting measure is one component of the unique vector
of measures that minimes an energy functional defined on admissible
vectors of measures. In addition, we show that each of the other
components is the limiting measure of the normalized counting
measure on certain generalized eigenvalues.
\end{abstract}

\pagestyle{myheadings} \thispagestyle{plain} \markboth{MAURICE DUITS
AND ARNO B.J. KUIJLAARS}{EIGENVALUES OF BANDED TOEPLITZ MATRICES}

\section{Introduction} \label{section1}

For an integrable function $a :  \{z\in
\mathbb C \mid |z|=1\} \to \mathbb C$ defined on the unit circle in the complex plane,
the $n\times n$ Toeplitz matrix $T_n(a)$ with symbol $a$ is
defined by
\begin{equation}
   \big(T_n(a)\big)_{jk}=a_{j-k},
    \qquad j,k = 1, \ldots, n,
\end{equation}
where $a_k$ is the $k$th Fourier coefficient of $a$,
\begin{equation}
   a_k =\frac{1}{2\pi} \int_0^{2\pi} a({\rm e}^{{\rm i} \theta}){\rm e}^{-{\rm i} k\theta}\ {\rm d}\theta.
\end{equation}
In this paper we study banded Toeplitz matrices for which the symbol has
only a finite number of non-zero Fourier coefficients. We
assume that there exist $p,q\geq 1$ such that
\begin{equation} \label{symbola}
   a(z)=\sum_{k=-q}^pa_kz^k, \qquad a_p \neq 0, \quad a_{-q} \neq 0.
\end{equation}
Thus  $T_n(a)$ has at most $p+q+1$ non-zero diagonals.
As in \cite[p.~263]{Bottcher-Grudsky}, we also assume without
loss of generality that
\begin{equation} \label{assumption}
  \textrm{g.c.d. } \{k \in \mathbb Z \mid a_k\neq0\}=1.
\end{equation}

We are interested in the limiting behavior of the spectrum of $T_n(a)$
as $n \to \infty$. We use $\spec T_n(a)$ to denote the
spectrum of $T_n(a)$:
\[ \spec T_n(a) = \{ \lambda \in \mathbb C \mid
    \det( T_n(a)-\lambda I) = 0\} \]
Spectral properties of banded Toeplitz
matrices are the topic of the recent book \cite{Bottcher-Grudsky} by
B\"ottcher and Grudsky. We  will refer to this book frequently,
in particular to Chapter 11 where the limiting behavior of the
spectrum is discussed.

The limiting behavior  of $\spec T_n(a)$ was characterized
by Schmidt and Spitzer \cite{Schmidt-Spitzer}. They
considered the set
\begin{equation}   \liminf_{n\to \infty} \spec T_n(a),
\end{equation}
consisting of all $\lambda\in \mathbb C$ such that there exists a
sequence $\{\lambda_n\}_{n\in \N}$, with $\lambda_n\in
\spec T_n(a)$, converging to $\lambda$, and the set
\begin{equation}
  \limsup_{n\to \infty} \spec T_n(a),
\end{equation}
consisting of all $\lambda$ such that there exists a sequence
$\{\lambda_n\}_{n\in \N}$, with $\lambda_n\in \spec T_n(a)$, that has
a  subsequence converging to $\lambda$. Schmidt and Spitzer showed
that these two sets are equal and can be characterized in terms of
the algebraic equation
\begin{equation} \label{algebraic equation}
  a(z)-\lambda=\sum_{k=-q}^p a_k z^k -\lambda =0.
\end{equation}
For every $\lambda \in \mathbb C$ there are $p+q$ solutions for
(\ref{algebraic equation}), which we denote by $z_j(\lambda)$, for
$j=1,\ldots,p+q$. We order these solutions by absolute value, so
that
\begin{equation} \label{ordering in magnitude}
  0<|z_1(\lambda)|\leq |z_2(\lambda)|\leq \cdots \leq
  |z_{p+q}(\lambda)|.
\end{equation}
When all inequalities in (\ref{ordering in magnitude}) are strict
then the values $z_k(\lambda)$ are unambiguously defined. If
equalities occur then we choose an arbitrary numbering so that
(\ref{ordering in magnitude}) holds. The result by Schmidt and
Spitzer \cite{Schmidt-Spitzer}, \cite[Theorem
11.17]{Bottcher-Grudsky}, is that
\begin{equation}
  \liminf_{n\to \infty}\spec T_n(a) =
  \limsup_{n\to \infty}\spec T_n(a)=\Gamma_0
\end{equation}
where
\begin{equation} \label{defGamma0}
    \Gamma_0 := \{ \lambda \in \mathbb C \mid
        |z_q(\lambda)| = |z_{q+1}(\lambda)| \}.
\end{equation}
This result gives a description of the asymptotic location of the
eigenvalues. The eigenvalues accumulate on the set $\Gamma_0$, which
is known to be a disjoint union of a finite number of (open)
analytic arcs and a finite number of exceptional points
\cite[Theorem 11.9]{Bottcher-Grudsky}. It is also known that
$\Gamma_0$ is connected \cite{Ullman}, \cite[Theorem
11.19]{Bottcher-Grudsky}, and that $\mathbb C \setminus \Gamma_0$
need not be connected \cite[Theorem 11.20]{Bottcher-Grudsky},
\cite[Proposition 5.2]{BGpaper}. See \cite{Bottcher-Grudsky} for
many beautiful illustrations of eigenvalues of banded Toeplitz
matrices.

The limiting eigenvalue distribution was determined
by Hirschman \cite{Hirschman}, \cite[Theorem 11.16]{Bottcher-Grudsky}.
He showed that
there exists a Borel probability measure $\mu_0$ on $\Gamma_0$ such that the
normalized eigenvalue counting measure of $T_n(a)$ converges weakly to
$\mu_0$, as $n \to \infty$. That is,
\begin{equation}
  \frac{1}{n} \sum_{\lambda\in \spec T_n(a)} \delta_\lambda \to \mu_0,
\end{equation}
where in the sum each eigenvalue is counted according to its
multiplicity. The measure $\mu_0$ is absolutely continuous with
respect to the arclength measure on $\Gamma_0$ and has an analytic
density on each open analytic arc in $\Gamma_0$, which can be
explicitly represented in terms of the solutions of the algebraic
equation (\ref{algebraic equation}) as follows. Equip every open
analytic arc in $\Gamma_0$ with an orientation. The orientation
induces $\pm$-sides on each arc, where the $+$-side is on the left
when traversing the arc according to its orientation, and the
$-$-side is on the left. The limiting measure $\mu_0$ is then given
by
\begin{equation} \label{maat q}
    {\rm d}\mu_0(\lambda)=\frac{1}{2\pi {\rm i}}\sum_{j=1}^{q}
    \left(\frac{{z_j'}_+(\lambda)}{{z_j}_+(\lambda)} - \frac{{z_j'}_-(\lambda)}{{z_j}_-(\lambda)}
    \right) {\rm d} \lambda.
\end{equation}
where ${\rm d} \lambda$ is the complex line element on $\Gamma_0$
(taken according to the orientation), and where
${z_j}_\pm(\lambda)$, $\lambda \in \Gamma_0$, is the limiting value
of $z_j(\lambda')$ as $\lambda' \to \lambda$ from the $\pm$ side of
the arc.  These limiting values exist for every $\lambda \in
\Gamma_0$, with the possible exception  of the finite number of
exceptional points.

Note that the right-hand side of (\ref{maat q}) is a priori a
complex measure and it is not immediately clear that it is in fact a
probability measure. In the original paper \cite{Hirschman} and in
the book \cite[Theorem 11.16]{Bottcher-Grudsky}, the authors give a
different expression for the limiting density, from which it is
clear that the measure is non-negative.  We prefer to work with the
complex expression (\ref{maat q}), since it allows for a direct
generalization which we will need in this paper.

Note also that if we reverse the orientation on an arc in
$\Gamma_0$, then the $\pm$-sides are reversed. Since the complex
line element ${\rm d} \lambda$ changes sign as well, the expression
(\ref{maat q}) does not depend on the choice of orientation.

The following is a very simple example, which however serves as a motivation
for the results in the paper.

\begin{example} Consider the symbol $a(z) = z + 1/z$.
In this case we find that $\Gamma_0 = [-2,2]$
and $\mu_0$ is absolutely continuous with respect
to the Lebesgue measure and has density
\begin{equation}
  \frac{{\rm d} \mu_0(\lambda)}{{\rm d}\lambda}=\frac{1}{\pi \sqrt{4-\lambda^2}}, \qquad
  \lambda \in (-2,2).
\end{equation}
This measure is well-known in potential theory and is called the
arcsine measure or the equilibrium measure of $\Gamma_0$, see e.g.\ \cite{Saff-Totik}.
It  has
the property that it minimizes the energy functional $I$ defined by
\begin{equation}\label{energy in voorbeeld}
    I(\mu)= \iint \log \frac{1}{|x-y|} \ {\rm d} \mu(x) \ {\rm d}\mu(y),
\end{equation}
among all Borel probability measures $\mu$ on $[-2,2]$.
The measure $\mu_0$ is also characterized by the equilibrium condition
\begin{equation} \label{variatonal-condition}
    \int \log|x-\lambda| \ {\rm d}\mu_0(\lambda) = 0, \qquad x \in [-2,2],
\end{equation}
which is the Euler-Lagrange variational condition for the minimization
problem.
\end{example}

The fact that $\mu_0$ is the equilibrium measure of $\Gamma_0$ is
special for symbols $a$ with $p=q=1$. In that case one may think
of the eigenvalues of $T_n(a)$  as charged particles
on $\Gamma_0$, each eigenvalue having a total charge $1/n$, that repel each
other with logarithmic interaction. The particles seek to minimize
the energy functional \eqref{energy in voorbeeld}. As $n \to \infty$,
they distribute themselves according to $\mu_0$ and $\mu_0$ is
the  minimizer of \eqref{energy in voorbeeld} among
all probability measures supported on $\Gamma_0$.

The aim of this paper is to characterize $\mu_0$ for general symbols
$a$ of the form \eqref{symbola} also in terms of an equilibrium
problem from potential theory. The corresponding equilibrium problem
is more complicated since it involves not only the measure $\mu_0$,
but a sequence of $p+q-1$ measures
\[ \mu_{-q+1}, \ \mu_{-q+2}, \ \ldots, \ \mu_{-1}, \  \mu_0, \
    \mu_1, \  \ldots, \  \mu_{p-2},  \ \mu_{p-1} \]
that jointly minimize an energy functional.

\section{Statement of results} \label{section2}

\subsection{The energy functional}

To state our results we need to introduce some notions from
potential theory. Main references for potential theory in the
complex plane are \cite{Ransford} and \cite{Saff-Totik}.

We will mainly work with finite positive measures on $\mathbb C$,
but we will also use $\nu_1 - \nu_2$ where $\nu_1$ and $\nu_2$
are positive measures.
The measures need not have bounded support. If $\nu$ has unbounded
support then we assume that
\begin{equation}\label{cond: unbounded support}
    \int \log(1+|x|) \ {\rm d} \nu(x)<\infty.
\end{equation}
In that case the logarithmic energy of $\nu$ is defined as
\begin{equation}
    I(\nu)=\int \log \frac{1}{|x-y|} \ {\rm d} \nu(x) {\rm d} \nu (y)
\end{equation}
and $I(\nu) \in (-\infty, +\infty]$.

\begin{definition}
We define $\mathcal M_e$ as the collection of positive measures
$\nu$ on $\mathbb C$ satisfying \eqref{cond: unbounded support} and
having finite energy, i.e., $I(\nu) < +\infty$. For $c > 0$ we
define
\begin{equation}
     \mathcal M_e(c) = \{ \nu \in \mathcal M_e \mid  \nu(\mathbb C) = c \}.
\end{equation}
\end{definition}

The mutual energy $I(\nu_1,\nu_2)$ of two measures $\nu_1$ and $\nu_2$ is
\begin{equation}
    I(\nu_1,\nu_2) = \int \log \frac{1}{|x-y|} \ {\rm d} \nu_1(x) {\rm d} \nu_2 (y).
\end{equation}
It is well-defined and finite if $\nu_1, \nu_2 \in \mathcal M_e$
and in that case we have
\begin{equation}
    I(\nu_1 - \nu_2) = I(\nu_1) + I(\nu_2) - 2 I(\nu_1,\nu_2).
\end{equation}

If $\nu_1, \nu_2 \in \mathcal M_e(c)$ for some $c >0$, then
\begin{equation}\label{eq:pos-energy}
I(\nu_1-\nu_2)\geq 0,
\end{equation}
with equality if and only if $\nu_1= \nu_2$. This is a well-known
result if $\nu_1$ and $\nu_2$ have compact support \cite{Saff-Totik}.
For measures in $\mathcal M_e(c)$ with unbounded support, this is a
recent result of Simeonov \cite{Simeonov}, who obtained this from
a very elegant integral representation for
$I(\nu_1- \nu_2)$. It is a consequence of \eqref{eq:pos-energy}
that $I$ is strictly convex on $\mathcal M_e(c)$, since
\[ \begin{aligned}
    I\left( \frac{\nu_1 + \nu_2}{2} \right)
    & = \frac{1}{2} \left( I(\nu_1) + I(\nu_2) \right)
        - I\left(\frac{\nu_1-\nu_2}{2}\right) \\
        &
        \leq \frac{1}{2} \left( I(\nu_1) + I(\nu_2) \right),
        & \textrm{ for } \nu_1,\nu_2 \in \mathcal M_e(c),
        \end{aligned} \]
with equality if and only if $\nu_1 = \nu_2$.

Before we can state the equilibrium problem we also need to
introduce the sets
\begin{equation} \label{defGammak}
  \Gamma_k :=\{\lambda\in \mathbb C \mid
  |z_{q+k}(\lambda)|=|z_{q+k+1}(\lambda)|\}, \qquad  k=-q+1,\ldots,p-1,
\end{equation}
which for $k=0$ reduces to the definition \eqref{defGamma0} of $\Gamma_0$.
We will show that each $\Gamma_k$ is the disjoint union of a finite number of open
analytic arcs and a finite number of exceptional points.
All $\Gamma_k$ are unbounded, except for $\Gamma_0$ which is compact.

The equilibrium problem will be defined for a vector of measures
denoted by $\vec \nu=(\nu_{-q+1},\ldots,\nu_{p-1})$. The component
$\nu_k$ is a measure  on $\Gamma_k$  satisfying some additional
properties that are given in the following definition.

\begin{definition}
We  call a vector of measures $\vec{\nu} = (\nu_{-q+1}, \ldots,
\nu_{p-1})$ admissible if $\nu_k \in \mathcal M_e$, $\nu_k$ is
supported on $\Gamma_k$, and  \begin{equation} \label{norm-nuk}
        \nu_k(\Gamma_k) =
        \begin{cases}
            \frac{q+k}{q} & \text{ if } k \leq 0, \\[5pt]
            \frac{p-k}{p} & \text{ if } k \geq 0,
        \end{cases} \end{equation}
for every $k=-q+1, \ldots, p-1$.
\end{definition}

Now we are ready to state our first result. The proof is given in section \ref{section4}.

\begin{theorem} \label{theorem1} Let the symbol $a$ satisfy \eqref{symbola} and \eqref{assumption},
    and let the curves $\Gamma_k$ be defined as in \eqref{defGammak}.
For each $k\in \{-q+1,\ldots,p-1\}$, define the measure $\mu_k$ on
$\Gamma_k$ by
\begin{equation} \label{maat k}
    {\rm d}\mu_k(\lambda)=\frac{1}{2\pi {\rm i}}\sum_{j=1}^{q+k}
    \left(\frac{{z_j'}_+(\lambda)}{{z_j}_+(\lambda)} - \frac{{z_j'}_-(\lambda)}{{z_j}_-(\lambda)}
    \right) \ {\rm d} \lambda,
  \end{equation}
where ${\rm d}\lambda$ is the complex line element on each
analytic arc of $\Gamma_k$ according to a chosen orientation of
$\Gamma_k$ (cf.\ discussion after \eqref{maat q}). Then
\begin{enumerate}
  \item[\rm (a)] $\vec{\mu}=(\mu_{-q+1},\ldots,\mu_{p-1})$ is admissible.
  \item[\rm (b)] There exist constants $l_k$ such that
 \begin{align} \label{EulerL-1}
    2 \int \log |\lambda-x| \ {\rm d} \mu_k(x)
    =\int \log |\lambda-x| \ {\rm d} \mu_{k+1}(x)+
    \int \log |\lambda-x| \ {\rm d} \mu_{k-1}(x)+l_k, \end{align}
    for $k=-q+1,\ldots,p-1$, and $\lambda \in \Gamma_k$. Here we let
$\mu_{-q}$ and $\mu_{p}$ be the
    zero measures.
    \item[\rm (c)] $\vec{\mu}=(\mu_{-q+1},\ldots,\mu_{p-1})$ is the unique
    minimizer of  the energy functional $J$ defined by
    \begin{equation} \label{energyJ}
        J(\vec{\nu}) = \sum_{k=-q+1}^{p-1} I(\nu_k) -
        \sum_{k=-q+1}^{p-2} I(\nu_k, \nu_{k+1})
    \end{equation} for  admissible vectors of measures $\vec{\nu}=(\nu_{-q+1},\ldots,\nu_{p-1})$.
\end{enumerate}
\end{theorem}
The relations \eqref{EulerL-1} are the Euler-Lagrange variational
conditions for the minimization problem for $J$ among admissible
vectors of measures.

It may not be obvious that the energy functional \eqref{energyJ} is bounded
from below. This can be seen from the alternative representation
\begin{align}
  J(\vec \nu)=&\left(\frac{1}{q}+\frac{1}{p}\right)I(\nu_0) +
    \sum_{k=1}^{q-1} k(k+1) \  I\left(\frac{\nu_{-q+k}}{k} - \frac{\nu_{-q+k+1}}{k+1}\right)\nonumber \\
    &+\sum_{k=1}^{p-1} k(k+1) \ I\left(\frac{\nu_{p-k}}{k}- \frac{\nu_{p-k-1}}{k+1} \right).
        \label{energyJ-alt}
 \end{align}
We leave the calculation leading to this identity to the reader.
Under the normalizations \eqref{norm-nuk} it follows by
\eqref{eq:pos-energy} that each term in the two finite sums on the right-hand side of
\eqref{energyJ-alt} is non-negative, so that
\begin{align} \nonumber
    J(\vec{\nu}) & \geq  \left(\frac{1}{q}+\frac{1}{p}\right) I(\nu_0).
\end{align}
Since $\nu_0$ is a Borel probability measure on $\Gamma_0$ and
$\Gamma_0$ is compact, we indeed have that the energy functional is
bounded from below on admissible vectors of measures $\vec \nu$.

The alternative representation \eqref{energyJ-alt} will  play a role
in the proof of Theorem \ref{theorem1}.

Yet another representation for $J$ is
\begin{equation} \label{energyJ-alt2}
    J(\vec\nu) =  \sum_{j,k=-q+1}^{p-1} A_{jk} \ I(\nu_j, \nu_k) \end{equation}
where the interaction matrix $A$ has entries
\begin{equation} \label{interaction} A_{jk} = \begin{cases}
    1, & \text{ if } j = k, \\
    -\frac{1}{2}, & \text{ if } |j-k| = 1, \\
    0, & \text{ if } |j-k| \geq 2.
    \end{cases} \end{equation}
The energy functional in the form  \eqref{energyJ-alt2} and
\eqref{interaction} also appears in the theory of simultaneous
rational approximation, where it is the interaction matrix for a
Nikishin system \cite[Chapter 5]{Nikishin-Sorokin}.

It allows for the following physical interpretation: on
each of the curves $\Gamma_k$ one puts charged particles with total
charge $(q+k)/q$ or $(p-k)/p$, depending on whether $k \leq 0$
or $k \geq 0$. Particles that lie on the same curve
repel each other. The particles on two consecutive curves interact
in the sense that they attract each other but in a way that is half as
strong as the repulsion on a single curve. Particles on different curves
that are not consecutive do not interact with each other in a direct way.

\subsection{The measures $\mu_k$ as limiting measures of generalized eigenvalues}

By \eqref{maat q} and Theorem \ref{theorem1} we know that the
measure $\mu_0$ that appears in the minimizer of the energy
functional $J$ is the limiting measure for the eigenvalues of
$T_n(a)$. It is natural to ask about the other measures $\mu_k$ that
appear in the minimizer. In our second result we show that the
measures $\mu_k$ can be obtained as limiting counting measures for
certain generalized eigenvalues.

Let $k \in \{-q+1, \ldots, p-1\}$. We use $T_n(z^{-k}(a-\lambda)$
to denote the Toeplitz matrix with the symbol $z \mapsto z^{-k}(a(z) - \lambda)$.
For example, for $k=1$, $q=1$ and $p=2$, we have
\[ T_n(z^{-k}(a-\lambda)) = \begin{pmatrix}
    a_1 & a_0 - \lambda & a_{-1} & \\
    a_2 & a_1 & a_0 - \lambda & a_{-1} \\
        & a_2 & a_1  & a_0-\lambda & a_{-1} \\
        &  & \ddots & \ddots & \ddots & \ddots \\
        & & & a_2 & a_1 & a_0-\lambda & a_{-1} \\
        &  &  & & a_2 & a_1 & a_0 - \lambda  \\
        &  &   &  &  & a_2 & a_1
          \end{pmatrix}_{n \times n}. \]

\begin{definition}
For $k \in \{-q+1, \ldots, p-1\}$ and $n \geq 1$, we define
the polynomial $P_{k,n}$ by
\begin{equation} \label{defPkn}
    P_{k,n}(\lambda)=\det T_n(z^{-k}(a-\lambda))
\end{equation}
and we define the $k$th generalized spectrum of $T_n(a)$ by
\begin{equation}
    \spec_k T_n(a) = \{ \lambda \in \mathbb C \mid P_{k,n}(\lambda) = 0
    \}.
    \end{equation}
Finally, we define $\mu_{k,n}$ as the normalized zero counting measure
of $\spec_k T_n(a)$
\begin{equation} \label{defmukn}
    \mu_{k,n} = \frac{1}{n} \sum_{\lambda \in \spec_k T_n(a)}  \delta_{\lambda}
    \end{equation}
where in the sum each $\lambda$  is counted according to its multiplicity
as a zero of $P_{k,n}$.
\end{definition}

Note that $\lambda \in \spec_k T_n(a)$ is a generalized eigenvalue
(in the usual sense) for the matrix pencil $(T_n(z^{-k}a),
T_n(z^{-k}))$, that is, $\det(A - \lambda B) = 0$ with $A =
T_n(z^{-k}a)$ and $B = T_n(z^{-k})$. If $k= 0 $, then $B=I$ and
$\spec_0 T_n(a)=\spec T_n(a)$. If $k\neq 0$, then $B$ is not
invertible and the generalized eigenvalue problem is singular,
causing that there are less than $n$ generalized eigenvalues. In
fact, since $T_n(z^{-k}(a-\lambda))$ has exactly $n-|k|$ entries
$a_0-\lambda$, we easily get that the degree of $P_{k,n}$ is at most
$n - |k|$ and so there are at most $n -|k|$ generalized eigenvalues.
Due to the band structure of $T_n(z^{-k}(a-\lambda))$ the actual
number of generalized eigenvalues is substantially smaller.

\begin{proposition} \label{eigenschappen Pnk}
  Let $k\in \{-q+1,\ldots,p-1\}$. Let
  $ P_{k,n}(\lambda) = \gamma_{k,n} \lambda^{d_{k,n}} + \cdots$
  have degree $d_{k,n}$ and leading coefficient $\gamma_{k,n} \neq 0$.
  Then
  \begin{equation} \label{degree Pnk}
    d_{k,n}\leq \begin{cases}
       \frac{q+k}{q} n, & \quad \textrm{ if } k<0,\\[5pt]
       \frac{p-k}{p} n, & \quad \textrm{ if } k>0.
    \end{cases}
  \end{equation}
Equality holds in \eqref{degree Pnk} if either $k > 0$ and $n$ is a multiple of $p$,
or $k < 0$ and $n$ is a multiple of $q$, and in those cases we have
\begin{align} \label{leadingcoefficient}
    \gamma_{k,n}=
      \begin{cases}
    (-1)^{(k+1)n} a_{-q}^{|k|n/q}, & \textrm{ if } k<0 \textrm{ and } n\equiv0 \bmod  q,\\[5pt]
    (-1)^{(k+1)n} a_p^{kn/p}, & \textrm{ if } k>0 \textrm{ and } n\equiv 0
    \bmod  p.
    \end{cases}
  \end{align}

\end{proposition}

We now come to our second main result. It is the analogue of the
results of Schmidt-Spitzer and Hirschman for the generalized
eigenvalues.
\begin{theorem} \label{theorem3}
Let $k\in \{-q+1,\ldots,p-1\}$. Then
\begin{equation} \label{p en q limiet van eigenwaarde}
  \liminf_{n\to \infty} \spec_k T_n(a) =
  \limsup_{n\to \infty} \spec_k T_n(a) = \Gamma_k,
\end{equation}
and
  \begin{equation} \label{eq: th: convergence of measures}
  \lim_{n\to \infty} \int_{\mathbb C} \phi(z) \ {\rm d}
  \mu_{k,n}(z)=\int_{\mathbb C} \phi(z) \ {\rm d}\mu_k(z)
  \end{equation}
    holds  for every bounded continuous  function $\phi$ on $\mathbb C$.
\end{theorem}

The key element in the proof of Theorem \ref{theorem3} is a beautiful formula
of Widom \cite{Widom}, see \cite[Theorem 2.8]{Bottcher-Grudsky}, for the
determinant of a banded Toeplitz matrix. In the present situation
Widom's formula yields the following.
Let $\lambda \in \mathbb C$ be such that the solutions $z_j(\lambda)$ of
the algebraic equation \eqref{algebraic equation} are mutually distinct.
Then
\begin{equation} \label{eq: Widom in b is zmink a min lambda}
 P_{k,n}(\lambda)= \det T_n(z^{-k} (a-\lambda))=\sum_M C_M(\lambda) \left(w_M(\lambda)\right)^n,
\end{equation}
where the sum is over all subsets $M \subset \{1,2,\ldots, p+q\}$ of
cardinality $|M|=p-k$ and for each such $M$, we have
\begin{equation} \label{defwM}
  w_M(\lambda):= (-1)^{p-k} a_p \prod_{j\in M} z_j(\lambda),
\end{equation}
and (with $\overline{M}:=\{1,2,\ldots,p+q\}\setminus M$),
\begin{equation} \label{defCM}
  C_M(\lambda):=\prod_{j\in M} z_j(\lambda)^{q+k} \prod_{j\in M \atop l\in
  \overline{M}} (z_j(\lambda)-z_l(\lambda))^{-1}.
\end{equation}
The formula \eqref{eq: Widom in b is zmink a min lambda} shows that
for large $n$, the main contribution comes from those $M$ for
which $|w_M(\lambda)|$ is the largest possible. For
$\lambda \in \mathbb C \setminus \Gamma_k$ there is a unique
such $M$, namely
\begin{equation} \label{defMk}
    M = M_k := \{ q+k+1, q+k+2, \ldots, p+q\}
\end{equation}
because of the ordering \eqref{ordering in magnitude}.

\subsection{Overview of the rest of the paper}
In section \ref{section3} we will state some preliminary results about
analyticity properties of the solutions $z_j$ of the algebraic
equation \eqref{algebraic equation}. These results will be needed in
the proof of Theorem \ref{theorem1} which is given in section \ref{section4}. In
section \ref{section5} we will prove Proposition \ref{eigenschappen Pnk} and
Theorem \ref{theorem3}. Finally, we conclude the paper by giving
some examples in section \ref{section6}.

\section{Preliminaries} \label{section3}

In this section we collect a number of properties of the curves
$\Gamma_k$ and the solutions $z_1(\lambda),\ldots,z_{p+q}(\lambda)$ of
the algebraic equation \eqref{algebraic equation}.
For convenience we define throughout the rest of the paper
\[ \Gamma_{-q} = \Gamma_p = \emptyset, \qquad \textrm{ and } \qquad
 \mu_{-q} = \mu_{p} = 0.  \qquad \textrm{ (the zero-measure)}. \]
Occasionally we also use
\[ z_0(\lambda) = 0, \qquad z_{p+q+1}(\lambda) = +\infty. \]

\subsection{The structure of the curves $\Gamma_k$}

We start with a definition, cf.\ \cite[\S 11.2]{Bottcher-Grudsky}.
\begin{definition}
    A point $\lambda_0 \in \mathbb C$ is called a branch point
    if $a(z) - \lambda_0 = 0$ has a multiple root.
    A point $\lambda_0 \in \Gamma_k$ is an exceptional point of $\Gamma_k$
    if $\lambda_0$ is a branch point, or if there is no open neighborhood
    $U$ of $\lambda$ such that $\Gamma_k \cap U$ is an analytic arc starting
    and terminating on $\partial U$.
\end{definition}

If $\lambda_0$ is a branch point, then  there is a $z_0$ such that
$a(z_0) = \lambda_0$ and $a'(z_0) = 0$. Then we may assume that $z_0
= z_{q+k}(\lambda_0) = z_{q+k+1}(\lambda_0)$ for some $k$ and
$\lambda_0 \in \Gamma_k$. For a symbol $a$ of the form
\eqref{symbola}, the derivative $a'$ has exactly $p+q$ zeros
(counted with multiplicity), so that there are exactly $p+q$ branch
points counted with multiplicity.

The solutions $z_k(\lambda)$ also have branching at infinity
(unless $p=1$ or $q=1$).
There are $p$ solutions of \eqref{algebraic equation}
that tend to infinity as $\lambda \to \infty$, and $q$ solutions that tend to $0$.
Indeed, we have
\begin{equation} \label{asympzk}
    z_k(\lambda) = \left\{ \begin{array}{ll}
     c_k \lambda^{-1/q} (1+ \OO(\lambda^{-1/q})),
     & \textrm{ for } k =1, \ldots, q, \\[10pt]
     c_k \lambda^{1/p} (1+ \OO(\lambda^{-1/p})),
     & \textrm{ for } k = q+1, \ldots, p+q,
     \end{array} \right. \end{equation}
as $\lambda \to \infty$. Here $c_1, \ldots, c_q$ are
the $q$ distinct solutions of $c^q = a_{-q}$ (taken in some order
depending on $\lambda$), and $c_{q+1}, \ldots, c_{p+q}$
are the $p$ distinct solutions of $c^p = a_p^{-1}$ (again taken
in some order depending on $\lambda$).

The following proposition gives the structure of $\Gamma_k$ at infinity.
\begin{proposition} \label{prop2}
    Let $k \in \{-q+1, \ldots, p-1\} \setminus \{0\}$.
    Then there is an $R > 0$ such that
    $\Gamma_k \cap \{ \lambda \in \mathbb C \mid |\lambda| > R \}$ is a finite disjoint
    union of analytic arcs, each extending from $|\lambda| = R$ to infinity.
\end{proposition}
\begin{proof}
The proof is similar to the proof of \cite[Proposition 11.8]{Bottcher-Grudsky}
where a similar structure theorem was proved for finite branch points.
We omit the details.
\end{proof}

It follows from Proposition \ref{prop2} that the exceptional
points for $\Gamma_k$ are in a bounded set. Since the set of exceptional
point is discrete we conclude that there are only finitely many
exceptional points. Then  we have the following
result about the structure of $\Gamma_k$.

\begin{proposition} \label{prop3}
For every $k \in \{-q+1, \ldots, p-1\}$, the set $\Gamma_k$
is the disjoint union of a finite number of open analytic arcs
and a finite number of exceptional points. The set $\Gamma_k$ has no isolated points.
\end{proposition}
\begin{proof}
This was proved for $k=0$ in \cite{Schmidt-Spitzer} and
\cite[Theorem 11.9]{Bottcher-Grudsky}.
For general $k$, there are only finitely
many exceptional points and the proof follows in a similar way.
\end{proof}

\subsection{The Riemann surface}

From Proposition \ref{prop3} it follows that the curves $\Gamma_k$
can be taken as cuts for the $p+q$-sheeted Riemann surface of the algebraic
equation \eqref{algebraic equation}. We number the sheets from $1$ to $p+q$, where
the $k$th sheet of the Riemann surface is
\begin{equation} \label{defRk}
    \mathcal R_k =  \{\lambda \in \mathbb C \mid |z_{k-1}(\lambda)| < |z_k(\lambda)| < |z_{k+1}(\lambda)| \}
    = \mathbb C \setminus (\Gamma_{-q+k-1} \cup \Gamma_{-q+k}). \end{equation}
Thus $z_k$ is well-defined and analytic on $\mathcal R_k$.

The easiest case to visualize is the case where consecutive cuts are disjoint,
that is, $\Gamma_{-q+k-1} \cap \Gamma_{-q+k} = \emptyset$ for every $k = 2, \ldots, p+q-2$.
In that case we have that $\mathcal R_k$ is connected to $\mathcal R_{k+1}$ via $\Gamma_{-q+k}$
in the usual crosswise manner, and  $z_{k+1}$ is the analytic continuation of $z_k$ across
$\Gamma_{-q+k}$.

The general case is described in the following proposition.

\begin{proposition} \label{analytic-continuation}
Suppose $A$ is an open analytic arc such
that $A \subset \Gamma_{-q+k}$, for $k = k_1, \ldots, k_2$,
and $A \cap (\Gamma_{-q + k_1 -1} \cup \Gamma_{-q + k_2+1}) = \emptyset$.
Then for $k=k_1, \ldots, k_2+1$, we have that the analytic continuation
of $z_k$ across $A$ is equal to $z_{k_1+k_2-k+1}$. Thus across $A$, we have
that $\mathcal R_k$ is connected to $\mathcal R_{k_1+k_2-k+1}$.
\end{proposition}
\begin{proof}
We have that
\[ |z_{k_1}(\lambda)| = |z_{k_1+1}(\lambda)| = \cdots
    = |z_{k_2}(\lambda)| = |z_{k_2+1}(\lambda)| \]
for $\lambda \in A$, with strict inequalities ($<$) for $\lambda$ on either side
of $A$. Choose an orientation for $A$. Then there is a permutation
$\pi$ of $\{k_1, \ldots, k_2+1\}$ such that $z_{\pi(k)}$ is the
analytic continuation of $z_k$ from the $+$-side of $A$ to the
$-$-side of $A$.

Assume that there are $k, k' \in \{k_1, \ldots, k_2+1\}$ such that
$k < k'$ and $\pi(k) < \pi(k')$. Take a regular $\lambda_0 \in A$
and a small neighborhood $U$ of $\lambda_0$ such that
$A \cap U = \Gamma_{-q+k} \cap U = \Gamma_{-q+k'} \cap U$
and $A \cap U$ is an analytic arc starting and terminating on $\partial U$.
Then we have a disjoint union $U = U_+ \cup U_- \cup (A \cap U)$
where $U_+$ ($U_-$) is the part of $U$ on the $+$-side ($-$-side) of $A$.
The function $\phi$ defined by
\[ \phi(\lambda)  = \left\{ \begin{array}{ll}
     \frac{z_k(\lambda)}{z_{k'}(\lambda)}, & \textrm{ for } \lambda \in U_+, \\[5pt]
     \frac{z_{\pi(k)}(\lambda)}{z_{\pi(k')}(\lambda)}, & \textrm{ for } \lambda \in U_-,
     \end{array} \right. \]
has an analytic continuation to $U$, and satisfies
$|\phi(\lambda)| < 1$ for $\lambda \in U_+ \cup U_-$ and $|\phi(\lambda)| = 1$
for $\lambda \in A \cap U$. This contradicts the maximum principle for analytic
functions. Therefore $\pi(k) > \pi(k')$ for every $k, k' \in \{k_1, \ldots, k_2 +1\}$
with $k < k'$, and this implies that $\pi(k) = k_1 + k_2 - k+1$ for every $k = k_1, \ldots, k_2 + 1$,
and the proposition follows.
\end{proof}

\subsection{The functions $w_k(\lambda)$}

A major role is played by the functions $w_k$,
which for $k \in \{-q+1, \ldots, p-1\}$, are defined by
\begin{equation} \label{defwk}
    w_k(\lambda) = \prod_{j=1}^{q+k} z_j(\lambda),
    \qquad \textrm{ for } \lambda \in \mathbb C \setminus \Gamma_k.
    \end{equation}
Note that $w_k=(-1)^{p-k}a_p^{-1} w_{\{1,\ldots,k\}}$ in the
notation of \eqref{defwM}.
\begin{proposition} \label{analyticity wk}
    The function $w_k$ is analytic in $\mathbb C \setminus \Gamma_k$.
\end{proposition}
\begin{proof}
    Since $z_j$ is analytic on $\mathcal R_j = \mathbb C \setminus (\Gamma_{-q+j-1} \cup \Gamma_{-q+j})$,
    see \eqref{defRk}, we obtain from its definition that
    $w_k$ is analytic in $\mathbb C \setminus \bigcup_{j=1}^{k+q} \Gamma_{-q+j}$.
  Let $A$ be an analytic arc in  $\Gamma_{-q+j} \setminus \Gamma_{k}$ for some
  $j < k+q$. Choose an orientation on $A$. Since the arc is disjoint from $\Gamma_{k}$,
  we have that $z_{j+}(\lambda) = z_{\pi(j)-}(\lambda)$, for $\lambda \in A$ and $j=1, \ldots, q+k$, where
  $\pi$ is a permutation of $\{1, \ldots, q+k\}$.
  Since $w_k$ is symmetric in the $z_j$'s for $j=1, \ldots, q+k$, it then follows that
  \[ w_{k+}(\lambda) = w_{k-}(\lambda), \qquad \textrm{ for } \lambda \in A, \]
  which shows the analyticity in $\mathbb C \setminus \Gamma_{k}$ with the possible
  exception of isolated singularities at the exceptional points of $\Gamma_{-q+1}$,
  $\Gamma_{-q+2}$, \ldots, $\Gamma_{k-1}$. However, each $z_j$, and therefore also $w_k$, is bounded near
  such an exceptional point, so that any isolated singularity is removable.
\end{proof}

In the rest of the paper we make frequently use of the logarithmic
derivative $w_k'/w_k$ of $w_k$. By the fact that  $w_k$ does not
vanish on $\C\setminus \Gamma_k$ and Proposition \ref{analyticity
wk}, it follows that $w_k'/w_k$ is analytic in $\C\setminus
\Gamma_k$. By Proposition \ref{analytic-continuation} it moreover
has an analytic continuation across every open analytic arc
$A\subset\Gamma_k$. Near the exceptional points that are no branch
points $w_k'/w_k$ remains bounded. At the branch points it can
however have singularities of a certain order.
\begin{proposition} \label{prop6}
    Let $\lambda_0 \in \Gamma_k$ be a branch point of $\Gamma_k$.
    Then there exists an $m\in \N$ such that
    \begin{equation} \label{eq: prop6}
        \frac{w_k'(\lambda)}{w_k(\lambda)} =
            \OO\left((\lambda - \lambda_0)^{-m/(m+1)}\right),
    \end{equation}
    as $\lambda \to \lambda_0$ with $\lambda \in \mathbb C \setminus \Gamma_k$.
\end{proposition}
\begin{proof}
Let $1\leq j\leq q+k$. We investigate the behavior of
$z_j(\lambda)$ when $\lambda\to \lambda_0$ such that $\lambda$
remains in a connected component of $\C\setminus (\Gamma_{j-1}\cup
\Gamma_{j})$. Then $z_j(\lambda)\to z_0$ for some $z_0\in \C$ with
$a(z_0)=\lambda_0$. Let $m_0+1$ be the multiplicity of $z_0$ as a
solution of $a(z)=\lambda_0$. Then
\begin{equation}\label{eq: lem: a(z) near yj}
  a(z)=\lambda_0+ c_0(z-z_0)^{m_0+1}(1+\OO(z-z_0)), \qquad
  z\to z_0,
\end{equation}
for some nonzero constant $c_0$. Therefore,
\begin{align} \label{eq: zj near lambda0}
  z_j(\lambda) &=z_0+\OO((\lambda-\lambda_0)^{1/(m_0+1)}),
\end{align}
and
\begin{align}
     \label{eq: zjprime near lambda0}
    z_j'(\lambda) & = \OO((\lambda-\lambda_0)^{-m_0/(m_0+1)}),
\end{align}
for $\lambda\to \lambda_0$ such that $\lambda$ remains in the same
connected component of $\C\setminus (\Gamma_{j-1}\cup \Gamma_{j})$.
Let $m$ be the maximum of all the multiplicities of the
roots of $a(z)=\lambda_0$. Then it follows from
\eqref{eq: zj near lambda0} and \eqref{eq: zjprime near lambda0}
that
\[ \frac{z_j'(\lambda)}{z_j(\lambda)} = \OO((\lambda-\lambda_0)^{-m/(m+1)})
\]
as $\lambda \to \lambda_0$ with $\lambda \in \mathbb C \setminus \Gamma_k$.
Then we obtain \eqref{eq: prop6} in view of \eqref{defwk}.
\end{proof}

 We end this section by giving the asymptotics of
$w_k'/w_k$ for $\lambda \to \infty$.
\begin{proposition} \label{prop5}
    As $\lambda \to \infty$ with $\lambda \in \mathbb C \setminus \Gamma_k$,
     we have
    \begin{equation} \frac{w_k'(\lambda)}{w_k(\lambda)} =
        \begin{cases}
        -\frac{q+k}{q} \lambda^{-1}
        + \OO\left(\lambda^{-1 - 1/q} \right),
            & \textrm{ for } k = -q+1, \ldots, -1, \\[5pt]
        - \lambda^{-1} + \OO(\lambda^{-2}),
            & \textrm{ for } k =0, \\[5pt]
        - \frac{p-k}{p} \lambda^{-1}
        + \OO\left(\lambda^{-1-1/p} \right),
            & \textrm{ for } k = 1, \ldots, p-1. \end{cases}
        \end{equation}
\end{proposition}
\begin{proof}
  This follows directly from \eqref{asympzk} and \eqref{defwk}.
\end{proof}

\section{Proof of Theorem \ref{theorem1}} \label{section4}

We use the function $w_k$ introduced in \eqref{defwk}.  We define
$\mu_k$ by the formula \eqref{maat k} and we note that
\begin{equation} \label{defmuk}
    {\rm d} \mu_k(\lambda) = \frac{1}{2\pi {\rm i}} \left(\frac{w_{k+}'(\lambda)}{w_{k+}(\lambda)} -
    \frac{w_{k-}'(\lambda)}{w_{k-}(\lambda)}\right) {\rm d} \lambda.
    \end{equation}

\begin{proposition} \label{prop7}
For each $k=-q+1, \ldots, p-1$, we have that
$\mu_k$ is a measure on $\Gamma_k$ with total mass $\mu_k(\Gamma_k) = (q+k)/q$ if $k \geq 0$,
and $\mu_k(\Gamma_k) = (p-k)/p$ if $k \geq 0$.
\end{proposition}

\begin{proof}
We first show that $\mu_k$ is a measure, i.e., that it is non-negative on each
analytic arc of $\Gamma_k$.
Let $A$ be an analytic arc in $\Gamma_k$ consisting only of regular points.
Let $t \mapsto \lambda(t)$ be a parametrization of $A$ in the direction of
the orientation of $\Gamma_k$. Then
\begin{align*}
    {\rm  d} \mu_k(\lambda) & =
    \frac{1}{2\pi {\rm i}} \left(\frac{w_{k+}'(\lambda(t))}{w_{k+}(\lambda(t))} -
      \frac{w_{k-}'(\lambda(t))}{w_{k-}(\lambda(t))}\right) \lambda'(t) {\rm d} t \\
      & = \frac{1}{2\pi {\rm i}}
        \left(\frac{\rm d}{{\rm d}t} \log \frac{w_{k+}(\lambda(t))}{w_{k-}(\lambda(t))} \right) {\rm d}t.
        \end{align*}
To conclude that $\mu_k$ is non-negative on $A$, it is thus enough to show that
\begin{equation} \label{logderRe}
    \text{Re} \log  \frac{w_{k+}(\lambda)}{w_{k-}(\lambda)}  = 0, \qquad
    \textrm{ for } \lambda \in A,
    \end{equation}
and
\begin{equation} \label{logderIm}
    \text{Im} \log \frac{w_{k+}(\lambda)}{w_{k-}(\lambda)} \quad \textrm{ increases along } A.
\end{equation}
Since $|w_{k+}(\lambda)| = |w_{k-}(\lambda)|$ for $\lambda \in A$, we have \eqref{logderRe}
so that it only remains to prove \eqref{logderIm}.

There is a neighborhood $U$ of $A$ such that $U \setminus \Gamma_k$
has two components, denoted $U_+$ and $U_-$, where $U_+$ is on the
$+$-side of $\Gamma_k$ and $U_-$ on the $-$-side. It follows from
Proposition \ref{analytic-continuation} that $w_k$ has an analytic
continuation from $U_-$ to $U$, which we denote by $\hat{w}_k$, and
that $|w_k(\lambda)| < |\hat{w}_k(\lambda)|$ for $\lambda \in U_+$,
and equality $|w_{k+}(\lambda)| = |\hat{w}_k(\lambda)|$ holds for
$\lambda \in A$. Thus it follows that
\[ \frac{\partial}{\partial n} \text{Re} \log \left(\frac{w_{k}(\lambda)}{\hat{w}_k(\lambda)} \right) \leq 0,
    \qquad \textrm{ for } \lambda \in A, \]
where $\frac{\partial}{\partial n}$ denotes the normal derivative to $A$ in the direction of $U_+$.
Then by the Cauchy-Riemann equations we have that
$\text{Im} \log \left( \frac{w_{k+}(\lambda)}{\hat{w}_{k+}(\lambda)} \right)$ is
increasing along $A$. Since $\hat{w}_{k+}(\lambda) = w_{k-}(\lambda)$ for $\lambda \in A$,
we obtain \eqref{logderIm}. Thus $\mu_k$ is a measure.

Next we show that $\mu_k$ is a finite measure, which means that we have to
show that
\begin{equation} \label{density}
    \frac{w_{k+}'(\lambda)}{w_{k+}(\lambda)} -
    \frac{w_{k-}'(\lambda)}{w_{k-}(\lambda)}
\end{equation}
is integrable near infinity on $\Gamma_k$ and near every branch
point on $\Gamma_k$. This follows from Propositions \ref{prop5} and
\ref{prop6}. Indeed, from Proposition \ref{prop5} it follows that
\begin{equation} \label{Oatinfinity}
    \frac{w_{k+}'(\lambda)}{w_{k+}(\lambda)} -
    \frac{w_{k-}'(\lambda)}{w_{k-}(\lambda)} = \OO\left(\lambda^{-1 - \delta}\right)
        \qquad \textrm{ as } \lambda \to \infty, \ \lambda \in \Gamma_k.
\end{equation}
where $\delta = 1/q$ if $k < 0$ and $\delta = 1/p$ if $k > 0$. Since $\delta > 0$
we see that \eqref{density} is integrable near infinity.
For a branch point $\lambda_0$ of $\Gamma_k$, we have
from Proposition \ref{prop6} that there exist an $m \geq 1$ such that
\begin{equation} \label{Onearbranch}
    \frac{w_{k+}'(\lambda)}{w_{k+}(\lambda)} -
    \frac{w_{k-}'(\lambda)}{w_{k-}(\lambda)} = \OO\left((\lambda - \lambda_0)^{-m/(m+1)}\right)
        \qquad \textrm{ as } \lambda \to \lambda_0, \ \lambda \in \Gamma_k.
\end{equation}
This shows that \eqref{density} is integrable near every branch point.
Thus $\mu_k$ is a finite measure.

\begin{figure}[t]
\centering
\includegraphics[scale=0.6]{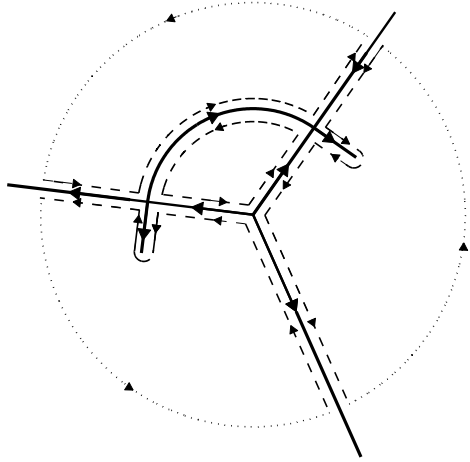}
\caption{Illustration for the proofs of Propositions \ref{prop7} and
\ref{prop8}.
The solid line is a sketch of a possible contour $\Gamma_k$. The dashed
line is the contour $\tilde \Gamma_{k,R}$ and the dotted line is the boundary
of a disk of radius $R$ around $0$.
} \label{Deforming the contour}
\end{figure}

Finally we compute the total mass of $\mu_k$. Let $D(0,R) = \{ z \in \mathbb C \mid |z| < R\}$.
Then for $R$ large enough, so that $D(0,R)$ contains all exceptional points of
$\Gamma_k$ and all connected components of $\mathbb C \setminus \Gamma_k$ (if any),
\begin{equation} \label{Gammakintegral}
    \mu_k(\Gamma_k \cap D(0,R))  =
    \frac{1}{2\pi {\rm i}} \left(\int_{\Gamma_k \cap D(0,R)}
    \frac{w_{k+}'(\lambda)}{w_{k+}(\lambda)} {\rm d}\lambda -
    \int_{\Gamma_k \cap D(0,R)}
    \frac{w_{k-}'(\lambda)}{w_{k-}(\lambda)} {\rm d}\lambda \right)
    \end{equation}
where we have used the behavior \eqref{Onearbranch} near the branch
points in order to be able to split the integrals. Again using
\eqref{Onearbranch} we can then turn the two integrals into a
contour integral over a contour $\tilde \Gamma_{k,R}$ as in
Figure \ref{Deforming the contour}. The contour $\tilde \Gamma_{k,R}$ passes
along the $\pm$-sides of $\Gamma_k \cap D(0,R)$ and if we choose
the orientation that is also shown in Figure \ref{Deforming the contour}
(and which is independent of the choice of orientation for $\Gamma_k$),
then
\begin{equation} \label{GammakRintegral}
    \mu_k(\Gamma_k \cap D(0,R))  =
    \frac{1}{2\pi {\rm i}} \int_{\tilde \Gamma_{k,R}}
        \frac{w_k'(\lambda)}{w_k(\lambda)} {\rm d}\lambda.
\end{equation}
The parts of $\tilde \Gamma_{k,R}$ that belong to bounded components
of $\mathbb C \setminus \Gamma_k$ form closed contours along the
boundary of each bounded component. By Cauchy's
theorem their contribution to the integral \eqref{GammakRintegral}
vanishes.
The parts of $\tilde \Gamma_{k,R}$ that belong to the unbounded
components of $\mathbb C \setminus \Gamma_k$ can be deformed
to the circle $\partial D(0,R)$ with the clockwise orientation.
Thus if we use the positive orientation on $\partial D(0,R)$
as in Figure \ref{Deforming the contour}, then we obtain from \eqref{GammakRintegral}
\[ \mu_k(\Gamma_k \cap D(0,R)) = - \frac{1}{2\pi {\rm i}}
    \oint_{\partial D(0,R)} \frac{w_k'(\lambda)}{w_k(\lambda)} {\rm d}\lambda \]
Letting $R \to \infty$ and using Proposition \ref{prop5}, we then find that
$\mu_k$ is a measure on $\Gamma_k$ with total mass $\mu_k(\Gamma_k)
= (q+k)/q$ if $k \leq 0$, and $\mu_k(\Gamma_k) = (p-k)/p$ if $k \geq 0$.
\end{proof}

The following proposition is the next step in showing
that the measures $\mu_k$ from \eqref{maat k} satisfy the
equations \eqref{EulerL-1}.

\begin{proposition} \label{prop8}
For $k=-q+1,\ldots,p-1$, we have that
\begin{equation} \label{Cauchytransform}
    \int \frac{{\rm d} \mu_k(x)}{x-\lambda} = \frac{w_k'(\lambda)}{w_k(\lambda)},
        \qquad \textrm{ for } \lambda \in \mathbb C \setminus \Gamma_k,
\end{equation}
and
  \begin{equation} \label{alphak-prop}
    \int \log |\lambda-x| \ {\rm d} \mu_k(x)=
    - \log |w_k(\lambda)| + \alpha_k, \qquad \textrm{ for } \lambda \in \mathbb C,
  \end{equation}
where $\alpha_k$ is the constant
\begin{equation} \label{alphak-def}
    \alpha_k=\left\{\begin{array}{ll}
     \log|a_{-q}| + \frac{k}{q} \log |a_{-q}|, & \quad \textrm{if } k \leq 0, \\[5pt]
     \log|a_{-q}|- \frac{k}{p}\log|a_{p}|, & \quad \textrm{if } k \geq 0.
    \end{array}\right.
  \end{equation}
\end{proposition}

\begin{proof}
To prove \eqref{Cauchytransform}, we follow the same arguments as in
the calculation of $\mu_k(\Gamma_k)$ in the end of the proof of Proposition
\ref{prop7}. Let $\lambda\in \C \setminus \Gamma_k$, and choose
$R> 0$ as in the proof of Proposition \ref{prop7}.
We may assume $R > |\lambda|$.
Then similar to \eqref{Gammakintegral} and \eqref{GammakRintegral}
we can write
\[ \int_{\Gamma_k \cap D(0,R)} \frac{{\rm d} \mu_k(x)}{x-\lambda}
    = \frac{1}{2\pi {\rm i}} \int_{\tilde \Gamma_{k,R}}
        \frac{w_k'(x)}{w_k(x)(x-\lambda)} {\rm d}x \]
where $\tilde \Gamma_{k,R}$ has the same meaning as in the
proof of Proposition \ref{prop7}, see also Figure \ref{Deforming the contour}.
As in the proof of Proposition \ref{prop7} we deform to
an integral over $\partial D(0,R)$, but now we have to take into account that the
integrand has a pole at $x=\lambda$ with residue
$w_k'(\lambda)/w_k(\lambda)$. Therefore, by Cauchy's theorem
\begin{align}
    \int_{\Gamma_k \cap D(0,R)}\frac{{\rm d} \mu_k(x)}{x-\lambda}
      &= \frac{w_k'(\lambda)}{w_k(\lambda)} -
      \frac{1}{2\pi {\rm i}} \int_{\partial D(0,R)}  \frac{w_{k}'(x)}{w_{k}(x)(x-\lambda)}
      {\rm d}x.
  \end{align}
Letting $R\to \infty$ and using
  Proposition \ref{prop5} gives \eqref{Cauchytransform}.

Next we integrate \eqref{Cauchytransform} over a Jordan curve $J$ in
$\mathbb C \setminus \Gamma_k$ from $\lambda_1$ to $\lambda_2$.
 \begin{align} \nonumber
   \int_{\lambda_1}^{\lambda_2} \int_{\Gamma_k}& \frac{1}{x-\lambda}  \ {\rm d}\mu_k(x)\
 {\rm d}\lambda
    =- \int \int_{\lambda_1}^{\lambda_2 } \frac{1}{x-\lambda} \ {\rm d}\lambda \ {\rm d}\mu_k(x)\\
    &=\int \left(  \log|\lambda_1-x|-\log|\lambda_2-x| + {\rm i} \Delta_J [\arg(\lambda-x)]
    \right) {\rm d}\mu_k(x),
    \label{integralJ1}
 \end{align}
 where $\Delta_J [\arg(\lambda-x)]$ denotes the change in argument of $\lambda-x$
 as when $\lambda$ varies over $J$ from $\lambda_1$ to $\lambda_2$.
By \eqref{Cauchytransform} the integral \eqref{integralJ1} is equal to
\begin{align} \label{integralJ2}
  \int_{\lambda_1}^{\lambda_2}  \frac{w_k'(\lambda)}{w_k(\lambda)} {\rm d} \lambda
  & = \log| w_k(\lambda_2)| - \log |w_k(\lambda_1)| +
    {\rm i} \Delta_J [\arg w_k(\lambda)].
\end{align}
Equating the real parts of \eqref{integralJ1} and \eqref{integralJ2} we get
\begin{equation}
    \int \left( \log|\lambda_1-x|-\log|\lambda_2-x| \right) {\rm d}\mu_k(x)
    = -\log|w_k(\lambda_1)| + \log |w_k(\lambda_2)|.
\end{equation}
Since $\lambda_1$ and $\lambda_2$ can be taken arbitrarily in a connected
component of $\mathbb C \setminus \Gamma_k$, we find that there
exists a constant $\alpha_k \in \mathbb R$ (which a priori could depend on the
connected component) such that
\begin{equation} \label{link log+alpha_k}
  \int  \log|\lambda-x | \ {\rm  d}\mu_k(x)
  = -\log|w_k(\lambda)| +\alpha_k,
\end{equation}
for all $\lambda$ in a connected component of $\mathbb C \setminus \Gamma_k$. By
continuity the equation \eqref{link log+alpha_k} extends to the closure of
the connected component, which shows that the same constant $\alpha_k$
is  valid for all connected components. Thus \eqref{link log+alpha_k} holds
for all $\lambda \in \mathbb C$.

The exact value of $\alpha_k$ can then
be determined by expanding \eqref{link log+alpha_k} for large $\lambda$.
Suppose for example that $k < 0$. Then by \eqref{asympzk} and \eqref{defwk}
\[ |w_k(\lambda)| = \prod_{j=1}^{q+k} |z_j(\lambda)|
    = |a_{-q}|^{(q+k)/q} |\lambda|^{-(q+k)/q} \left( 1 + \OO(\lambda^{-1/q})\right)
\]
as $\lambda \to \infty$. Thus
\begin{equation} \label{logwk}
    - \log |w_k(\lambda)| = \frac{q+k}{q} \log |\lambda| - \frac{q+k}{q} \log |a_{-q}|
    + \OO(\lambda^{-1/q}). \end{equation}
Since
\begin{equation} \label{logmuk}
    \int \log |\lambda - x| \ {\rm d} \mu_k(x)
    = \log |\lambda| \mu_k(\Gamma_k) + o(1)
    = \frac{q+k}{q} \log |\lambda| + o(1),
    \end{equation}
as $\lambda\to \infty$, the value \eqref{alphak-def} for $\alpha_k$
follows from \eqref{link log+alpha_k}, \eqref{logwk}, and
\eqref{logmuk}. The argument for $k > 0$ is similar. This completes
the proof of the proposition.
\end{proof}

To prove part (c) of Theorem \ref{theorem1} we
also need the following lemma.

\begin{lemma} \label{lem: nulenergie voor tekenvectormaten}
  Let $\vec \nu_1=(\nu_{1,-q+1}\ldots,\nu_{1,p-1})$ and
  $\vec \nu_2=(\nu_{2,-q+1}\ldots,\nu_{2,p-1})$ be two admissible
  vectors of measures.
  Then $J(\vec \nu_1-\vec \nu_2)$ is well defined and
  \begin{equation} \label{eq: nulenergie voor tekenvectormaten}
    J(\vec \nu_1-\vec \nu_2)\geq 0,
  \end{equation}
  with equality if and only if $\vec \nu_1=\vec \nu_2$.
\end{lemma}
\begin{proof}
Since both $\vec{\nu}_1$ and $\vec{\nu}_2$ have finite energy, we
find that $J(\vec{\nu}_1-\vec{\nu}_2)$ is well defined. According
to the alternative representation \eqref{energyJ-alt}, we have
\begin{align}
  J(\vec \nu _1 -\vec \nu_2)=
  &\left(\frac{1}{q}+\frac{1}{p}\right)I(\nu_{1,0}-\nu_{2,0})\nonumber\\
  & +\sum_{k=1}^{q-1} k(k+1)
 I\left(  \frac{\nu_{1,-q+k}}{k} - \frac{\nu_{2,-q+k}}{k}
     - \frac{\nu_{1, -q+k+1}}{k+1}  + \frac{\nu_{2,-q+ k+1}}{k+1}
     \right) \nonumber \\
  &+\sum_{k=1}^{p-1} k(k+1) I\left( \frac{\nu_{1,p-k}}{k} - \frac{\nu_{2,p-k}}{k}
     - \frac{\nu_{1, p-k-1}}{k+1} + \frac{\nu_{2,p-k-1}}{k+1} \right).
     \label{Jenergy-difference}
 \end{align}
 Using \eqref{eq:pos-energy} and \eqref{norm-nuk}, we see that all terms in
 \eqref{Jenergy-difference} are non-negative and therefore \eqref{eq: nulenergie voor
 tekenvectormaten} holds.

 Suppose now that $J(\vec \nu_1-\vec \nu_2)=0$. Then all terms in
 the right-hand side of
 \eqref{Jenergy-difference} are zero, so that
 \begin{align} \label{nu0equality}
   \nu_{1,0} & = \nu_{2,0}, \\
   \label{nu1equality}
   \frac{\nu_{1,-q+k}}{k}  + \frac{\nu_{2,-q+ k+1}}{k+1}
    & = \frac{\nu_{1, -q+k+1}}{k+1} + \frac{\nu_{2,-q+k}}{k},
    \qquad \textrm{for } k=1, \ldots, q-1, \\
  \label{nu2equality}
    \frac{\nu_{1,p-k}}{k} + \frac{\nu_{2,p-k-1}}{k+1}
   & = \frac{\nu_{1, p-k-1}}{k+1} + \frac{\nu_{2,p-k}}{k},
    \qquad \textrm{ for } k = 1,\ldots, p-1.
  \end{align}
Using \eqref{nu0equality} in \eqref{nu1equality} with $k=q-1$,
we find $\nu_{1,-1} = \nu_{2,-1}$. Proceeding inductively
we then obtain from \eqref{nu1equality} that $\nu_{1,k} = \nu_{2,k}$
for all $k = -q+1, \ldots, 0$. Similarly, from \eqref{nu0equality} and
\eqref{nu2equality} it follows that $\nu_{1,k} = \nu_{2,k}$ for
$k=0,\ldots, p-1$, so that $\vec \nu_1 = \vec \nu_2$ as claimed.
\end{proof}

Now we are ready for the proof of Theorem \ref{theorem1}.

\textit{Proof of  Theorem \ref{theorem1}.}
(a) In view of Proposition \ref{prop7} it
only remains to show that $\mu_k \in \mathcal M_e$ for every
$k=-q+1, \ldots, p-1$. The decay estimate \eqref{Oatinfinity}
implies that
\[ \int \log(1+|\lambda|) \ {\rm d}\mu_k(\lambda) < \infty. \]
The fact that $I(\mu_k) < +\infty$ follows from \eqref{alphak-prop}.
Indeed,
\begin{align*}
    I(\mu_k) & = - \iint  \log |\lambda - x| {\rm d}\mu_k(x) {\rm d}\mu_k(\lambda)
     = \int (\log |w_k(\lambda)| - \alpha_k) {\rm d}\mu_k(\lambda)
\end{align*}
and this is finite since $\mu_k$ is a finite measure on $\Gamma_k$
with a density that decays as in \eqref{Oatinfinity} and $\log
|w_k(\lambda)|$ is continuous on $\Gamma_k$ and  grows only as a
constant times $\log |\lambda|$ as $\lambda \to \infty$. Thus $\vec\mu$ is admissible
and part (a) is proved.

(b) According to \eqref{alphak-prop} we have
\begin{align} \nonumber
 & 2 \int  \log|\lambda-x| \ {\rm  d}\mu_k(x)
 - \int \log |\lambda-x| \ {\rm d} \mu_{k+1}(\lambda) - \int \log |\lambda-x| \ {\rm d} \mu_{k-1}(\lambda) \\
 & \nonumber \qquad = - 2\log |w_k(\lambda)| + 2\alpha_k
    + \log |w_{k+1}(\lambda)| - \alpha_{k+1} + \log |w_{k-1}(\lambda)| - \alpha_{k-1} \\
 & \nonumber \qquad = \log \left| \frac{w_{k+1}(\lambda) w_{k-1}(\lambda)}{w_k(\lambda)^2} \right|
    + 2\alpha_k - \alpha_{k+1} - \alpha_{k-1} \\
 & \qquad = \log \left| \frac{z_{q+k+1}(\lambda)}{z_{q+k}(\lambda)} \right|
    + 2\alpha_k - \alpha_{k+1} - \alpha_{k-1}.
    \label{EulerL-2}
  \end{align}
Since $|z_{q+k}(\lambda)| = |z_{q+k+1}(\lambda)|$ for $\lambda \in
\Gamma_k$, we see from \eqref{EulerL-2} that \eqref{EulerL-1} holds
with constant
\begin{equation} \label{deflk}
    l_k = 2\alpha_k-\alpha_{k-1}+\alpha_{k+1}.
\end{equation}
Note that for $k = -q+1$ and $k=p-1$, we are using the convention
that $\mu_{-q} = \mu_p = 0$, and we also have put $\alpha_{-q} =
\alpha_p = 0$. This proves part (b).

(c) Let $\vec{\nu} = (\nu_{-q+1}, \ldots, \nu_{p-1})$ be any
admissible vector of measures.
From the representation \eqref{energyJ-alt2} we get
\begin{align} \nonumber
  J(\vec{\nu}) & = J(\vec \mu + \vec{\nu} - \vec \mu) \\
  & = J(\vec \mu)+ J(\vec \nu-\vec\mu)+ 2\sum_{j,k=-q+1}^{p-1}
  A_{jk} I(\mu_j,\nu_k-\mu_k). \label{Jenergy2}
  \end{align}
Using \eqref{interaction}, we find from \eqref{Jenergy2}
\begin{align}
    J(\vec \nu)
  &=J(\vec \mu)+ J(\vec \nu-\vec\mu)+ \sum_{k=-q+1}^{p-1}
  I(2\mu_k-\mu_{k-1}-\mu_{k+1} ,\nu_k-\mu_k)
  \label{Jenergy3}
  \end{align}
For each $k = -q+1, \ldots, p-1$, we have
\begin{align} \nonumber
    & I(2\mu_k-\mu_{k-1}-\mu_{k+1},\nu_k-\mu_k) \\
    & \qquad = \label{Jenergy4}
    \int \left(
         \int \log |\lambda-x| \ {\rm d} (2 \mu_k- \mu_{k-1} - \mu_{k+1})(x) \right)
         d(\nu_k-\mu_k)(\lambda)
 \end{align}
By \eqref{EulerL-1} the inner integral in the right-hand side of
\eqref{Jenergy4} is constant for $\lambda \in \Gamma_k$. Since
$\nu_k$ and $\mu_k$ are finite measures on $\Gamma_k$ with
$\nu_k(\Gamma_k) = \mu_k(\Gamma_k)$, we find from \eqref{Jenergy4}
that
\[ I(2\mu_k-\mu_{k-1}-\mu_{k+1},\nu_k-\mu_k) = 0,
    \qquad \textrm{ for } k = -q+1, \ldots, p-1. \]
Then  \eqref{Jenergy3} shows that $J(\vec{\nu})=J(\vec \mu)+ J(\vec
\nu-\vec\mu)$, which by Lemma \ref{lem: nulenergie voor
tekenvectormaten} implies that $J(\vec \nu) \geq J(\vec \mu)$ and
equality holds if and only if $\vec \nu = \vec \mu$. This completes the
proof of Theorem~\ref{theorem1}. $\square$

\section{Proofs of Proposition \ref{eigenschappen Pnk} and Theorem \ref{theorem3}}
\label{section5}

\subsection{Proof of Proposition \ref{eigenschappen Pnk}}

We will now prove Proposition \ref{eigenschappen Pnk}, which follows
by a combinatorial argument.

\textit{Proof of Proposition \ref{eigenschappen Pnk}.}  We prove
\eqref{degree Pnk} and \eqref{leadingcoefficient} for $k>0$.
The case $k<0$ is similar. Let us first expand the determinant in
the definition of $P_{k,n}$
\begin{align} \label{Pkn expanded determinant}
    P_{k,n}(\lambda) =
    \det T_n(z^{-k}(a-\lambda))=
    \sum_{\pi \in S_n} \prod_{j=1}^n
    (a-\lambda)_{j-\pi(j)+k}.
\end{align}
Here $S_n$ denotes the set of all permutation on $\{1,\ldots,n\}$.
By the band structure of $T_n(z^{-k}(a-\lambda))$ it follows that we
only have non-zero contributions from permutations $\pi$ that satisfy
\begin{align} \label{condition on permutations}
    k- p \leq \pi(j)-j  \leq q+k,
    \qquad \textrm{ for all } j=1, \ldots, n.
\end{align}
Define for $\pi \in S_n$,
\begin{equation} N_\pi=\{j\ | \ \pi(j)=j+k\}.
\end{equation}
and denote the number of elements of
$N_\pi$ by $|N_\pi|$. For each $\pi \in S_n$ we have that
$\prod_{j=1}^n  (a-\lambda)_{j-\pi(j)+k}$ is a polynomial in $\lambda$ of degree
at most $|N_\pi|$.  So by \eqref{Pkn expanded determinant}
\begin{equation}\label{towards upperbound dkn}
d_{k,n} = \deg P_{k,n} \leq \max_\pi |N_\pi|
\end{equation}  where we maximize over permutations $\pi \in S_n$ satisfying
\eqref{condition on permutations}.

Let $\pi \in S_n$ satisfying \eqref{condition on permutations}. We
prove \eqref{degree Pnk} by giving an upper bound for $|N_\pi|$.
Since $\sum_{j=1}^n (\pi(j)-j)=0$ we obtain
\begin{equation} \label{eq: positivepartsperm}
 \sum_{j=1}^n (\pi(j)-j)_+= \sum_{j=1}^n (j-\pi(j))_+,
\end{equation}
where $(\cdot)_+$ is defined as $(a)_+=\max(0,a)$ for $a\in \R$.
Each  $j\in N_\pi$ gives a contribution $k$ to the left-hand side of
\eqref{eq: positivepartsperm}. Therefore the left-hand side is
at least $k|N_\pi|$. By \eqref{condition on permutations}
we have that each term in the right hand side is at most $p-k$.
Moreover, there are at most $n-|N_\pi|$ non-zero terms in this sum.
Combining this with \eqref{eq: positivepartsperm} leads to
\begin{equation}\label{eq: upperbound Npi1}
  k|N_\pi|\leq \sum_{j=1}^n (\pi(j) - j)_+ =
    \sum_{j=1}^n (j-\pi(j))_+ \leq (n-|N_\pi|)(p-k).
\end{equation}
Hence, if $\pi$ is a permutation satisfying \eqref{condition on permutations}
\begin{equation} \label{eq: upperbound Npi2}
  |N_\pi| \leq \frac{n(p-k)}{p}.
\end{equation}
Now \eqref{degree Pnk} follows by combining \eqref{eq: upperbound
Npi2} and \eqref{towards upperbound dkn}.

To prove \eqref{leadingcoefficient}, we assume that $n\equiv 0
\bmod p$. We claim that there exists a unique $\pi$ such that
equality holds in \eqref{eq: upperbound Npi2}. Then equality
holds in both inequalities of \eqref{eq: upperbound Npi1} and the above arguments show that
this can only happen if
\begin{align} \label{eq: toward equility in n0modp}
  \pi(j)=j+k,\qquad \textrm{ or } \qquad \pi(j)=j-p+k,
\end{align}
for every $j = 1, \ldots, n$. We claim that there exists a unique
such permutation, namely
\begin{equation} \label{permutation highest degree}
  \pi(j)=\begin{cases}
    j+k,& \textrm{ if } j\equiv 1,\ldots, (p-k) \bmod p,\\
    j-p+k,& \textrm{ if } j\equiv (p-k+1), \ldots, p \bmod p.
  \end{cases}
\end{equation}

To see this let  $\pi$ be a permutation satisfying \eqref{eq: toward
equility in n0modp}. The numbers $1,\ldots,p-k$ can not satisfy
$\pi(j)=j-p+k$ and thus satisfy $\pi(j) = j+k$. On the other hand, the
numbers $1,\ldots,k$ can not be the image of numbers $j$ satisfying $\pi(j) = j+k$,
and thus $\pi(j)=j-p+k$ for $j = p-k+1,\ldots,p$. So
\eqref{permutation highest degree} holds for $j = 1, \ldots, p$.
This means in particular that the restriction of $\pi$ to
$\{p+1,\ldots,n\}$ is again a permutation, but now on
$\{p+1,\ldots,n\}$.  By the same arguments we then find that
\eqref{permutation highest degree} holds for
$j = p+1,\ldots,2p$, and so on. The result is that
\eqref{permutation highest degree} is indeed the only permutation
that satisfies \eqref{eq: toward equility in n0modp}.

Finally, a straightforward calculation shows that the coefficient of
$\lambda^{(p-k)n/p}$ in $ \prod_{j=1}^n (a-\lambda)_{j-\pi(j)+k}$
with $\pi$ as in \eqref{permutation highest degree}  is nonzero and
given by \eqref{leadingcoefficient}. This proves the proposition.
 $\square$

\subsection{Proof of Theorem \ref{theorem3}}

Before we start with the proof of Theorem \ref{theorem3} we first
prove the following proposition concerning the asymptotics for
$P_{k,n}$ for $n\to \infty$.
  \begin{proposition} \label{prop9}
  Let $M_k=\{q+k+1,\ldots, p+q\}$.
We have that
 \begin{equation}\label{eq:  Toeplitz asymptotics}
P_{k,n}(\lambda)=(w_{M_k}(\lambda))^n C_{M_k}(\lambda)\left(1+\OO(\exp(-c_Kn)\right),
    \qquad n\to \infty,
\end{equation}
uniformly on compact subsets $K$ of $\mathbb C\setminus
\Gamma_k$. Here $c_K$ is a positive constant depending on $K$.
\end{proposition}
\begin{proof}
First rewrite \eqref{eq: Widom in b is zmink a min lambda} as
\begin{equation}\label{eq: prop:  Toeplitz asymptotics}
    P_{k,n}(\lambda)=(w_{M_k}(\lambda))^n C_{M_k}(\lambda)\left(1+R_{k,n}(\lambda)\right).
\end{equation}
with $R_{k,n}$ defined by \begin{equation}\label{eq: term to estimate}
R_{k,n}(\lambda)=\sum_{M\neq M_k}
\frac{(w_M(\lambda))^nC_M(\lambda)}{(w_{M_k}(\lambda))^nC_{M_k}(\lambda)}.\end{equation} Let $K$ be a
compact subset of $\mathbb C\setminus \Gamma_k$. If $K$ does not
contain branch points then there exists $A,B>0$ such that
\begin{equation}\label{eq: estimates on Cm} A<|C_M(\lambda)|<B\end{equation} for all $\lambda \in K$ and $M$.
Moreover, we have
 \begin{equation} \left|\frac{w_M(\lambda)}{w_{M_k}(\lambda)}\right|\leq
\left|\frac{z_{q+k}(\lambda)}{z_{q+k+1}(\lambda)}\right|\leq
\sup_{\lambda \in K}
\left|\frac{z_{q+k}(\lambda)}{z_{q+k+1}(\lambda)}\right|<1,
\end{equation}
for all $\lambda \in K$ and $M\neq M_k$. Therefore one readily
verifies from \eqref{eq: prop:  Toeplitz asymptotics}
that there exist $c_K$ such that $|R_{k,n}(\lambda)|\leq \exp(-
c_Kn )$  for all $\lambda \in K$ and $n$ large enough. This proves
the statement in case $K$ does not contain branch points.

Suppose that  $K$ does contain branch points.  Without loss of
generality we can assume that all branch points lie in the interior
of $K$ (otherwise we replace $K$ by a  bigger compact set). The
boundary $\partial K$ of $K$ is a compact set with no branch points
and therefore  \eqref{eq: Toeplitz asymptotics} holds for $\partial
K$ by the above arguments. Since $w_{M_k}$ and $C_{M_k}$
are analytic in $K$, we find by \eqref{eq: prop: Toeplitz
asymptotics} that $R_{k,n}$ is analytic in $K$. The maximum modulus
principle for analytic functions states that $\sup_{z\in K}|R_{k,n}(z)|=\sup_{z\in \partial K}
|R_{k,n}(z)|$ and thereby we obtain that  \eqref{eq: Toeplitz
asymptotics} also holds for $K$ with the same constant $c_K=c_{\partial K}$.
\end{proof}

We now state two particular consequences of \eqref{eq:  Toeplitz
asymptotics}.

\begin{corollary}\label{cor1}
Let $k\in \{-q+1,\ldots,p-1\}$. For every compact set $K\subset \C
\setminus \Gamma_k$ we have that  $\mu_{k,n}(K)=0$ for $n$ large
enough.
\end{corollary}
\begin{proof}
Let $K$ be a compact subset of $\C\setminus \Gamma_k$.  By
\eqref{eq:  Toeplitz asymptotics} it follows that $P_{k,n}$ has no
zeros in $K$ for large $n$. Since $n\mu_{k,n}(K)$ equals the number
of zeros of $P_{k,n}$ in $K$ the corollary follows.
\end{proof}

\begin{corollary}\label{cor2}
Let $k\in \{-q+1,\ldots,p-1\}$. We have that
\begin{align}
  \label{eq: convergence of Cauchy-transforms}
  \lim_{n\to \infty} \int_{\C} \frac{ {\rm
  d}\mu_{k,n}(x)}{x-\lambda}&=\int_{\Gamma_k}
  \frac{ {\rm d}\mu_k(x)}{x-\lambda},
\end{align}
 uniformly on compact subsets of $\mathbb C\setminus \Gamma_k$.
\end{corollary}
\begin{proof}
Let $K$ be a compact subset of $\C \setminus \Gamma_k$. Note that
\begin{equation}
  \int \frac{{\rm d}\mu_{k,n}(x)}{x-\lambda}  =\frac{1}{n}\sum_{\lambda_i
  \in \spec_k
T_n(a)}\frac{1}{\lambda_i-\lambda}=- \frac{P_{k,n}'(\lambda)}{n
  P_{k,n}(\lambda)}, \label{eq: cor: conv meas 1}
\end{equation}
for all $\lambda \in K$. With $M_k$ and $c_K$ as in Proposition
\ref{prop9} we obtain from \eqref{eq:  Toeplitz asymptotics} that
\begin{equation}  \label{eq: cor: conv meas 2}
  \frac{P_{k,n}'(\lambda)}{n P_{k,n}(\lambda)}=
  \frac{w'_{M_k}(\lambda)}{w_{M_k}(\lambda)}+\OO(1/n), \qquad n\to
  \infty,
\end{equation}
uniformly on $K$.
Let us rewrite the right-hand side of \eqref{eq: cor: conv meas 2}.
By expanding both sides of $z^{q}(a(z)-\lambda)=a_p\prod_{j=1}^{p+q}
(z-z_j(\lambda))$ and collecting the constant terms we obtain
\begin{equation}
  \prod_{j=1}^{p+q}(-z_j(\lambda))= \frac{a_{-q}}{a_p}.
\end{equation}
Since $\lambda \notin \Gamma_k$, we can split this product in two
parts, take the logarithmic derivative and use \eqref{defwk} and
\eqref{defwM} to obtain
\begin{equation} \label{eq: cor: conv meas 3}
0= \sum_{j=1}^{q+k} \frac{z_j'(\lambda)}{z_j(\lambda)}+
    \sum_{j=q+k+1}^{p+q} \frac{z_j'(\lambda)}{z_j(\lambda)}=
    \frac{w_k'(\lambda)}{w_k(\lambda)}+\frac{w_{M_k}'(\lambda)}{w_{M_k}(\lambda)}.
\end{equation}
Combining \eqref{eq: cor: conv meas 1}, \eqref{eq: cor: conv meas 2}
and \eqref{eq: cor: conv meas 3}, we obtain
\begin{equation} \label{Cauchylimit}
    \lim_{n\to \infty} \int \frac{ {\rm d}\mu_{k,n}(x)}{x-\lambda}=
    \frac{w_k'(\lambda)}{w_k(\lambda)}
\end{equation}
uniformly on $K$. Then \eqref{eq: convergence of Cauchy-transforms} follows from
    \eqref{Cauchylimit} and  \eqref{Cauchytransform}.
\end{proof}

Now we are ready for the proof of Theorem \ref{theorem3}.

\textit{Proof of Theorem \ref{theorem3}.}

First we prove (\ref{eq: th: convergence of measures}).
By Proposition \ref{eigenschappen Pnk} and the fact that
$\vec{\mu}$ is admissible, we get (see \eqref{norm-nuk})
\begin{equation}\label{upperbound measures}
    \mu_{k,n}(\C)=\frac{1}{n} \deg P_{k,n} \leq \mu_{k}(\C),
\end{equation}
for every $n\in \N$.

Let $C_0(\C)$ be the Banach space of continuous functions
on $\C$ that vanish at infinity. The dual space $C_0(\C)^*$
of $C_0(\C)$ is the space of regular complex Borel measures
on $\C$. By \eqref{upperbound measures} the sequence
$(\mu_{k,n})_{n \in \N}$ belongs to the ball in $C_0(\C)^*$ centered
at the origin with radius
$\mu_{k}(\C)$, which is weak$^*$ compact by the Banach-Alaoglu theorem.
Let $\mu_{k,\infty}$ be the limit of a weak$^*$ convergent
subsequence of $(\mu_{k,n})_{n\in \N}$.

By weak$^*$ convergence and Corollary \ref{cor1} we obtain that
 $\mu_{k,\infty}$ is supported on $\Gamma_k$.
Combining this with \eqref{eq: convergence of Cauchy-transforms} and
the weak$^{*}$ convergence leads to
\begin{align} \label{equality cauchy transforms}
\frac{1}{2\pi{\rm i}}\int_{\Gamma_k}
  \frac{{\rm d}\mu_k(x)}{x-\lambda} =
  \frac{1}{2\pi{\rm i}} \int_{\Gamma_k}
  \frac{{\rm d}\mu_{k,\infty}(x)}{x-\lambda},
\end{align}
for every $\lambda \in \C\setminus \Gamma_k$.  The integrals
in \eqref{equality cauchy transforms} are known in the literature as
the Cauchy transforms of the measures $\mu_k$ and $\mu_{k,\infty}$.
The Cauchy transform on $\Gamma_k$ is an injective map that maps
measures on $\Gamma_k$ to functions that are analytic in
$\C\setminus \Gamma_k$ (one can find explicit inversion
formulae, see for example the arguments in \cite[Theorem II.1.4]{Saff-Totik}
or the Stieltjes-Perron inversion formula in the special case
$\Gamma_k\subset \R$). Thus it follows from \eqref{equality cauchy transforms}
that $\mu_{k,\infty}=\mu_k$.
Therefore
\begin{equation} \label{weakstarconvergence}
    \lim_{n \to \infty} \mu_{k,n} = \mu_k
    \end{equation}
in the sense of weak$^*$ convergence in  $C_0(\C)^*$.
Thus \eqref{eq: th: convergence of measures} holds if $\phi$ is
a continuous function that vanishes at infinity.

From \eqref{upperbound measures} and \eqref{weakstarconvergence}
it also follows that
\begin{equation} \label{eq: no leaking}
  \lim_{n\to \infty} \mu_{k,n}(\C) = \mu_k(\C),
\end{equation}
Then the sequence $(\mu_{k,n})_{n\in \N}$ is tight. That is, for every $\eps>0$
there exists a compact $K$ such that $\mu_{k,n}(\C\setminus K)<\eps$
for every $n\in \N$. By a standard approximation argument one can now
show that \eqref{eq: th: convergence of measures} holds for
every bounded continuous function $\phi$ on $\C$.

Having \eqref{eq: th: convergence of measures} and Proposition \ref{prop9}, we can prove
\eqref{p en q limiet van eigenwaarde} as in
\cite[Theorem 11.17]{Bottcher-Grudsky}. Indeed, the sets $\liminf_{n\to \infty} \spec_k T_n(a)$
and $\limsup_{n\to\infty} \spec_k T_n(a)$ equal the support of $\mu_k$, which is $\Gamma_k$. $\square$

\section{Examples} \label{section6}

\subsection{Example 1}

As a first example consider the symbol $a$ defined by
\begin{equation}\label{symbol example 2}
  a(z)=\frac{4(z+1)^3}{27 z}.
\end{equation}
In this case we have $p=2$ and $q=1$. So we obtain two contours
$\Gamma_0$ and $\Gamma_1$ with two associated measures $\mu_0$ and
$\mu_1$. This example appeared  in
\cite{WalterenCoussementen}, in which the authors gave explicit
expressions for $\Gamma_0$ and $\mu_0$. The following proposition
also contains expressions for $\Gamma_1$ and $\mu_1$. In what
follows we take the principal branches for all fractional powers.
\begin{proposition} \label{propex1}
    With $a$ as in \eqref{symbol example 2}, we have that
  $\Gamma_0=[0,1]$ and
  \begin{equation} \label{ex: prop: densitymu0}
    {\rm
    d}\mu_0(\lambda)=\frac{\sqrt{3}}{4\pi}
    \frac{\left(1+\sqrt{1-\lambda}\right)^{1/3}+\left(1-\sqrt{1-\lambda}\right)^{1/3}}{ \lambda^{2/3}\sqrt{1-\lambda}} \ {\rm d} \lambda.
  \end{equation}
Moreover,  $\Gamma_1=(-\infty,0]$ and
\begin{equation}\label{ex: prop: densitymu1}{\rm
    d}\mu_1(\lambda)=\frac{\sqrt{3}}{{4\pi}}\frac{\left(1+\sqrt{1-\lambda}\right)^{1/3}-\left(\sqrt{1-\lambda}-1\right)^{1/3}}{ (-\lambda)^{2/3}\sqrt{1-\lambda}} \ {\rm d} \lambda.
    \end{equation}
\end{proposition}

\begin{figure}[t]
\centering
     \includegraphics[scale=0.25,angle=-90]{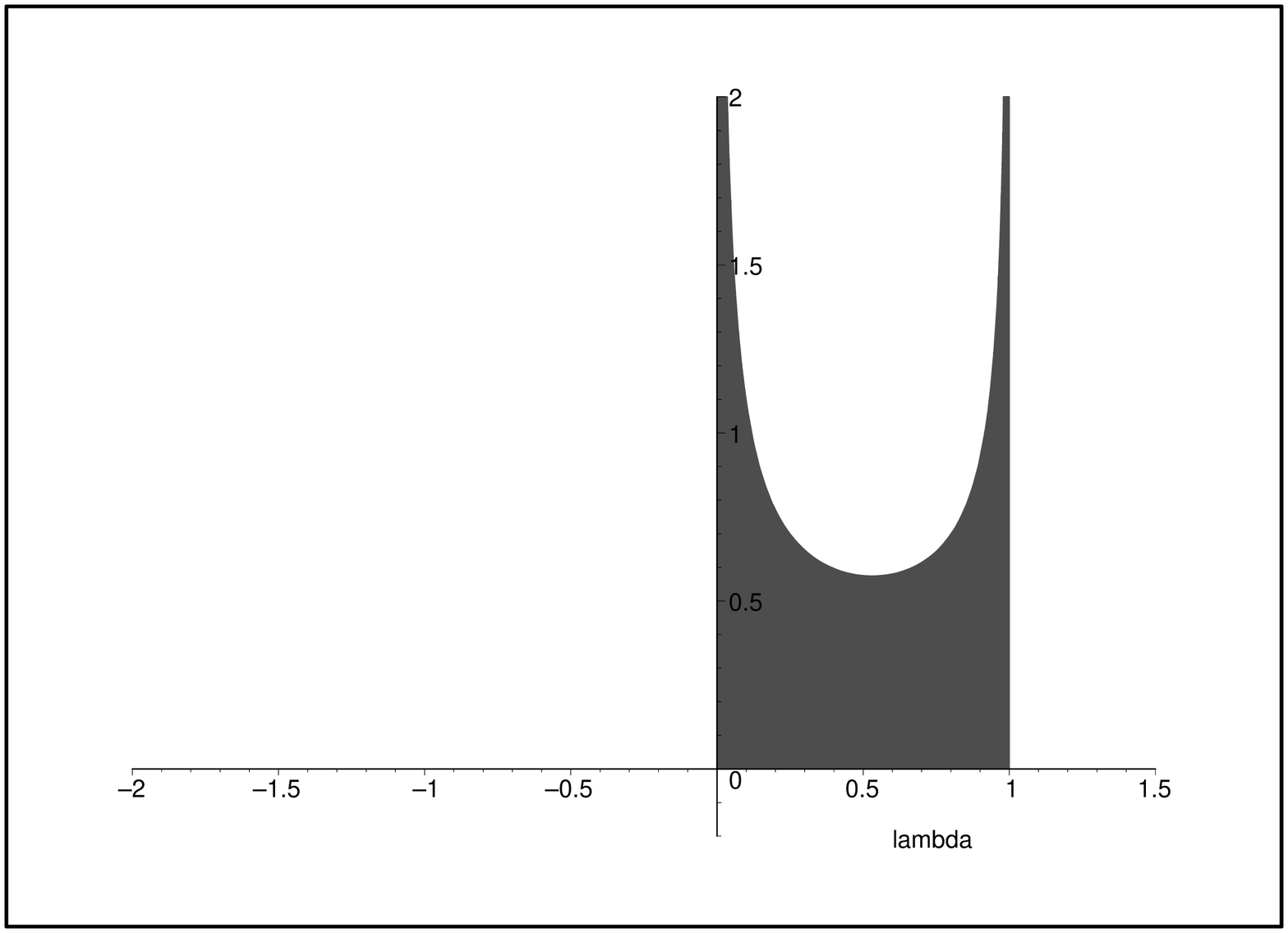} \quad 
     \includegraphics[scale=0.25,angle=-90]{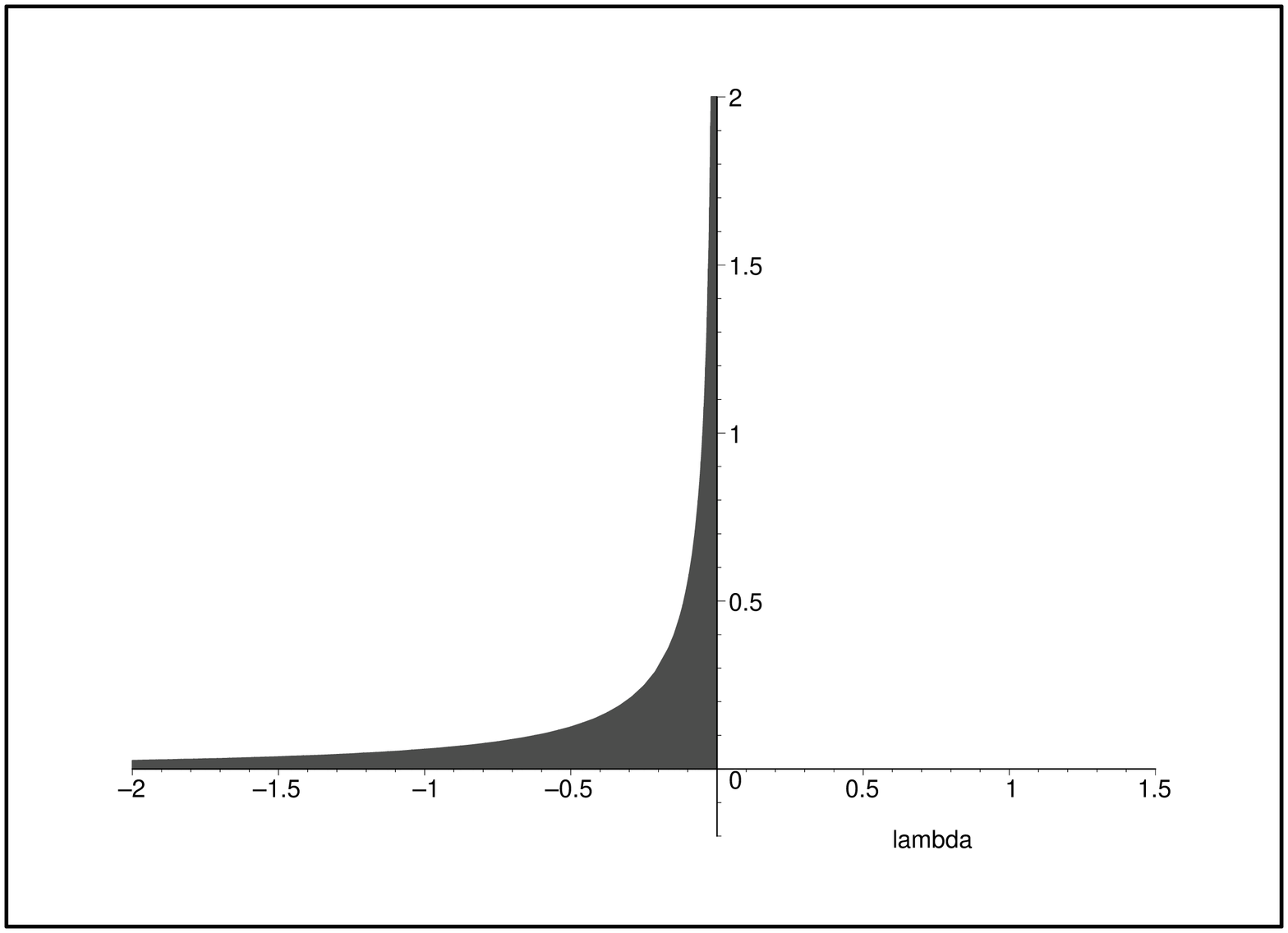}
    \caption{Illustration for Example 1: The densities of the measures $\mu_0$ (left) and $\mu_1$
    (right) for  $a=\frac{4(z+1)^3}{27 z}$.}
    \label{ex1: densities}
 \end{figure}

\begin{proof}
A straightforward calculation shows that
$\lambda=0$ and $\lambda=1$ are the branch points.

Let $\lambda\in \Gamma_0\cup \Gamma_1$ and assume that $\lambda$ is
not a branch point. There exist $y_1, y_2\in \C$ such that $y_1\neq
y_2$, $|y_1|=|y_2|$ and $a(y_1)=a(y_2)=\lambda$. Then it
follows from \eqref{symbol example 2} that
$|y_1+1|=|y_2+1|$. Therefore $y_1$ and $y_2$ are intersection
points of a circle centered at $-1$ and a circle centered at the
origin. Since $y_1\neq y_2$, this means that $y_1=\overline{y_2}$
and therefore $\lambda=a(y_1)=a(\overline{y_2})=\overline{a(y_1)}=\overline{\lambda}$,
so that $\lambda \in \R$. A further investigation shows that $a(z) - \lambda$
has $3$ different real zeros if $\lambda>1$. If $\lambda<1$ and
$\lambda\neq 0$ then $a(z) - \lambda$ has precisely $1$ real zero and $2$
conjugate complex zeros.  Therefore,
$\Gamma_0\cup\Gamma_1=(-\infty,1]$.

Now we will show that  $\Gamma_0=[0,1]$ and $\Gamma_1=(-\infty,0]$.
By Cardano's formula the solutions of the algebraic equation
$a(z)=\lambda$ are given by
\begin{equation}\label{voorbeeld 2 zjlambdagrdan 0}
  z_j(\lambda)=-1-\frac{3\lambda^{1/3}}2\left(\omega^j\left(1+(1-\lambda)^{1/2}\right)^{1/3}+\omega^{-j}
  \left(1-(1-\lambda)^{1/2}\right)^{1/3}\right),
\end{equation}
for $\lambda\in [0,1]$ and
\begin{equation}\label{voorbeeld 2 zjlambdakldan}
  z_j(\lambda)=-1+\frac{3(-\lambda)^{1/3}}2\left(\omega^{j+2}\left(1+(1-\lambda)^{1/2}\right)^{1/3}-\omega^{-j-2}
  \left((1-\lambda)^{1/2}-1\right)^{1/3}\right),
\end{equation}
for $\lambda\in(-\infty,0]$. Here $\omega={\rm e}^{2\pi {\rm i}/3}$. One can check that
$|z_1(\lambda)|=|z_2(\lambda)|<|z_3(\lambda)|$ for $\lambda \in
(0,1]$ and $|z_1(\lambda)|<|z_2(\lambda)|=|z_3(\lambda)|$ for
$\lambda \in (-\infty,0)$. Moreover, for  $\lambda=0$ we have
$z_1(0)=z_2(0)=z_3(0)=-1$.  Therefore $\Gamma_0=[0,1]$ and
$\Gamma_1=(-\infty,0]$.

\begin{figure}[t]
 \centering
    \includegraphics[scale=0.5]{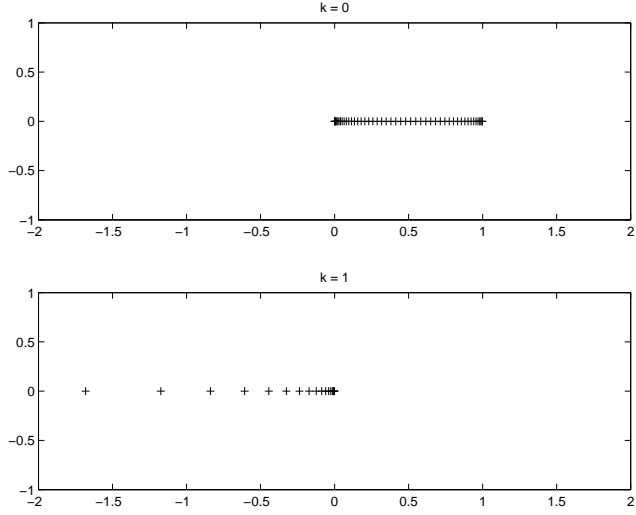}
    \caption{Illustration for Example 1: The spectrum $\spec T_{50}(a)$ (top) and the
    generalized spectrum ${\spec}_1 T_{50}(a)$ (bottom), for the symbol $a=\frac{4 (z+1)^3}{27
    z}$.}\label{ex1: eigenvalues}
  \end{figure}

The density \eqref{ex: prop: densitymu0} was already given in
\cite{WalterenCoussementen} and  \eqref{ex: prop: densitymu1}
follows in a similar way.
\end{proof}

In Figure \ref{ex1: densities} we plotted the densities of $\mu_0$
and $\mu_1$. Note that, due to the interaction between $\mu_0$ and
$\mu_1$ in the energy functional, there is more mass of $\mu_0$ near
$0$ than near $1$. We also see that the singularities of the densities
for $\mu_0$ and $\mu_1$ are of order $\OO(|\lambda|^{-2/3})$ for
$\lambda\to 0$, whereas the typical nature of a singularity in each
of the measures  is a square root singularity. The stronger singularity is
due to the fact that  $a(z)-\lambda$ has a triple root for $\lambda=0$.

In Figure \ref{ex1: eigenvalues} we plotted the eigenvalues and
generalized eigenvalues for $n=50$. It is known that the eigenvalues
are simple and positive  \cite[\S 2.3]{WalterenCoussementen}, which
we also see in Figure \ref{ex1: eigenvalues}.

\subsection{Example 2}

For the symbol $a$ defined by
\begin{equation}
  a(z)=z^2+z+z^{-1}+z^{-2}.
\end{equation}
we have $p=q=2$. From the symmetry $a(1/z)=a(z)$ it follows that
$\Gamma_{-1}=\Gamma_1$ and $\mu_{-1}=\mu_1$.

The interesting feature of this example is that the contours
$\Gamma_0$ and $\Gamma_{\pm 1}$ overlap. To be precise, the interval
$(-9/4,0)$ is contained in all three contours $\Gamma_{-1},\Gamma_0$
and $\Gamma_1$.  This can be most easily seen by investigating the
image of the unit circle under $a$. Consider
\begin{equation}
  a({\rm e}^{{\rm i} t})=2\cos 2t +2\cos t, \qquad \textrm{ for } t \in [0,2\pi).
\end{equation}
A straightforward analysis shows that for every
$\lambda \in (-9/4,0)$, the equation $a({\rm e}^{{\rm i} t}) =\lambda$
has four different solutions for $t$ in $[0,2\pi)$. This means that
the four solutions of the equation $a(z)=\lambda$ are on the unit circle,
and so in particular have the same absolute value.

The equation $a(z)-\lambda=0$ can be explicitly solved by
introducing the variable $y=z+1/z$. In exactly the same way as in
the previous example one can obtain the limiting measures. We will
not give the explicit formulas, but only plot the densities in
Figure \ref{ex3: densities}. The branch points are
$\lambda=-9/4$, $\lambda=0$ and $\lambda=4$. The contours are given by
\begin{equation}
    \Gamma_0  =[-9/4,4], \qquad
    \Gamma_{-1}  =\Gamma_1=(-\infty,0].
\end{equation}

\begin{figure}[t]
\centering
     \includegraphics[scale=0.25,angle=-90]{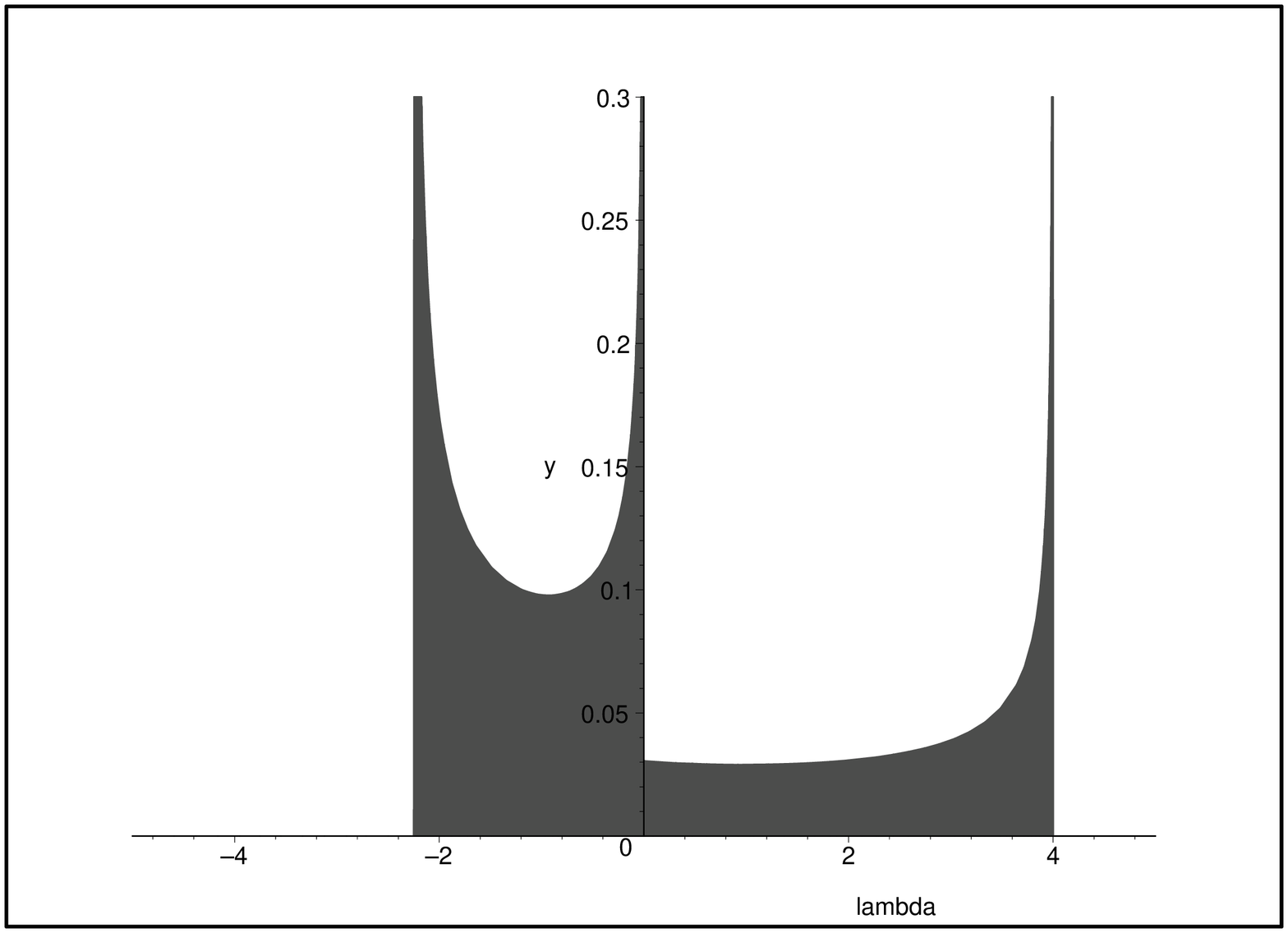} \quad 
     \includegraphics[scale=0.25,angle=-90]{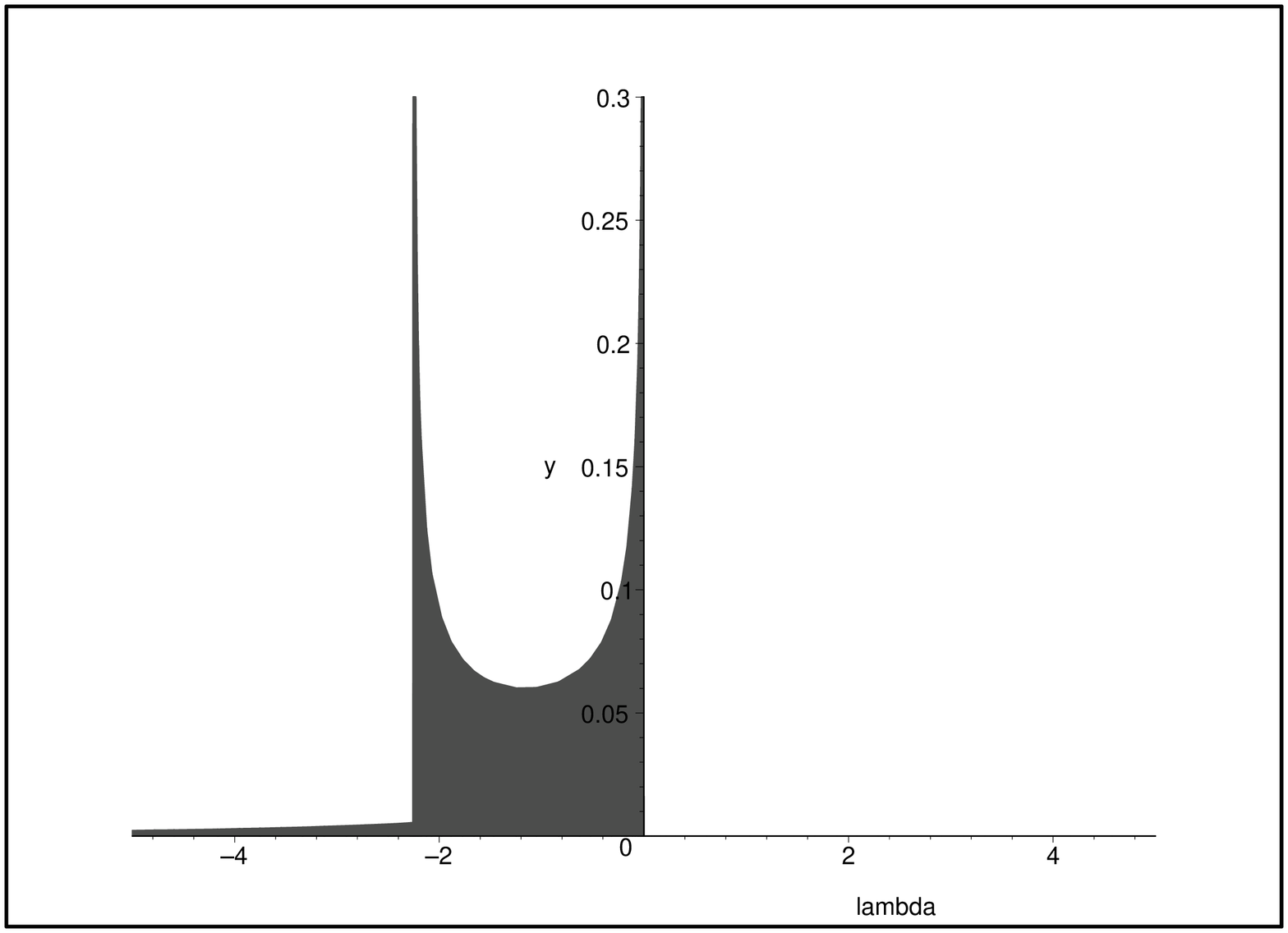} 
    \caption{Illustration for Example 2: The densities of the measures $\mu_0$ (left) and $\mu_1=\mu_{-1}$
    (right) for  $a(z)=z^2+z+z^{-1}+z^{-2}$.}
    \label{ex3: densities}
 \end{figure}

The densities have singularities at the branch points in
the interior of their supports. The singularities are only felt at one
side of the branch points. Consider first $\mu_0$, whose density
has a singularity at $0$. However the limiting value when
$0$ is approached from the positive real axis is finite. The change
in behavior of $\mu_0$ has to do with the fact that $z_1$ is
analytic on $(0,4)$ but not on $(-9/4,0)$. Therefore we find by
(\ref{maat q}) that
\begin{equation}
  {\rm d} \mu_0(\lambda)=\frac{1}{2\pi {\rm i}}
  \left(\frac{{z_1}_+'(\lambda)}{{z_1}_+(\lambda)}+\frac{{z_2}_+'(\lambda)}{{z_2}_+(\lambda)}-\frac{{z_1}_-'(\lambda)}{{z_1}_-(\lambda)}-\frac{{z_2}_-'(\lambda)}{{z_2}_-(\lambda)}\right){\rm
  d} \lambda
\end{equation}
on $(-9/4,0)$, and
\begin{equation}{\rm d}
\mu_0(\lambda)=\frac{1}{2\pi {\rm i}}
  \left(\frac{{z_2}_+'(\lambda)}{{z_2}_+(\lambda)}-\frac{{z_2}_-'(\lambda)}{{z_2}_-(\lambda)}\right){\rm
  d} \lambda
  \end{equation}
  on $(0,4)$.

For $\mu_{-1}=\mu_1$ a similar phenomenon happens at
$\lambda=-9/4$. This is a consequence of the fact that $z_1$ has an
analytic continuation into $z_2$ when we cross $(-\infty,-9/4)$, but
it has an analytic continuation into $z_4$ when we cross $(-9/4,0)$.

\subsection{Example 3}
As a final example, consider the symbol
\begin{equation}
  a(z)=z^p+z^{-q},
\end{equation}
with $p,q\geq 1$ and ${\textrm{gcd}} (p,q)=1$. This example
appeared in \cite{Schmidt-Spitzer}, where the authors  mentioned
that $\Gamma_0$ is given by the star
\begin{equation} \label{star}
  \Gamma_0=\{r \omega^j \mid  j=1,\ldots,p+q, \, 0\leq r \leq R\}
\end{equation}
with $ \omega={\rm e}^{ 2 \pi {\rm
  i}/(p+q)}$ and $R=(p+q)p^{-p/(p+q)}q^{-q/(p+q)}$. The other contours also have  a star shape,
  namely
  \begin{equation}
    \Gamma_k= \{(-1)^k  r \omega^j \mid  j=1,\ldots,p+q, \, 0\leq  r <\infty\}
 \end{equation}
 for $k\neq 0$. Note that the star $\Gamma_k$ for $k \neq 0$ is unbounded.

 In Figure \ref{figure ex3} we plotted the eigenvalues and the
generalized eigenvalues for $p=2$, $q=3$ and  $n=50$. All the
(generalized) eigenvalues appear to lie exactly on the contours.
In  the special case $p=1$ it is known that the eigenvalues of $T_n(a)$
lie indeed  precisely on the star \eqref{star} and
are all simple (possibly except for $0$)  \cite[Theorem 3.2]{eiermannvarga},
see also \cite{Kuijlaars} for a connection to Chebyshev-type quadrature.

\begin{figure}[t]
  \centering
  \includegraphics[scale=0.65]{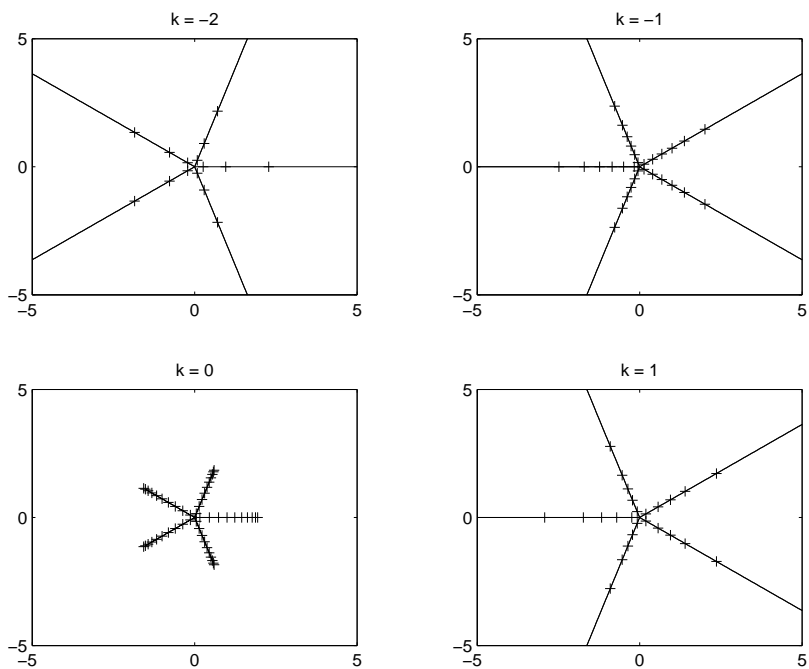}
  \caption{Illustration for Example 3: The contours $\Gamma_k$ and the eigenvalues  and
  generalized eigenvalues for  $T_{50}(a)$ for the symbol $a=z^2+z^{-3}$.}\label{figure ex3}
\end{figure}

\subsection{Numerical stability}

In Figure \ref{ex1: eigenvalues} and Figure \ref{figure ex3} the
eigenvalues and the generalized eigenvalues of $T_{50}(a)$ were
computed numerically. To control the stability of the numerical
computation of the eigenvalues one needs to analyze the
pseudo-spectrum. For banded Toeplitz matrices the pseudo-spectrum is
well understood  \cite[Th. 7.2]{trefethenembree}. To this date, a
similar analysis of the pseudo-spectrum for the matrix pencil
$(T_n(z^{-k}a),T_n(z^{-k}))$ has not been carried out. See \cite[\S
X.45]{trefethenembree} for some remarks on the pseudo-spectrum for
the generalized eigenvalue problem.

\bibliographystyle{amsplain}

\end{document}